\documentclass[11pt]{article}
\usepackage{amsthm, amsmath, amssymb, amsfonts, url, booktabs, tikz, setspace, fancyhdr, bm}
\usepackage{geometry}
\usepackage{hyphenat}
\usepackage{enumerate}
\usepackage[shortlabels]{enumitem}
\usepackage[babel]{microtype}
\usepackage[english]{babel}
\usepackage{comment}
\usepackage{bbm}
\usepackage{csquotes}
\usepackage{mathabx}
\usepackage{graphicx}
\usepackage{float}
\usepackage{xcolor}
\usepackage{mathtools}
\usetikzlibrary{positioning, arrows.meta, shapes.geometric}
\counterwithin{figure}{section}

\newtheorem{theorem}{Theorem}[section]

\newtheorem{lemma}[theorem]{Lemma}
\newtheorem{cor}[theorem]{Corollary}
\newtheorem{claim}[theorem]{Claim}

\usetikzlibrary{decorations.pathmorphing}

\theoremstyle{definition}
\newtheorem{defn}[theorem]{Definition}
\newtheorem*{defn-non}{Definition}

\newlist{Case}{enumerate}{2}
\setlist[Case, 1]{%
    label           =   {\bfseries Case \arabic*.},
    labelindent=1em ,labelwidth=1.3cm, labelsep*=1em, leftmargin =!
}
\setlist[Case, 2]{%
    label           =   {\bfseries Subcase \arabic{Casei}.\arabic*.},
    labelindent=-1em ,labelwidth=1.3cm, labelsep*=1em, leftmargin =!
}

\usepackage{hyperref}
\usepackage[capitalise]{cleveref}

\begin{document}
\title{\bfseries\Large Tur\'an number of a matching and a Berge triangle}
%Extremal Problems for Berge Triangles with Bounded Matching Number
%
%
%
%
%
\author{
Yue Xu\textsuperscript{1}\thanks{Email: \href{mailto:xuyue1@sjtu.edu.cn}{xuyue1@sjtu.edu.cn}},\quad
Jiasheng Zeng\textsuperscript{2}\thanks{Email: \href{mailto:jzengbl@connect.ust.hk}{jzengbl@connect.ust.hk}},\quad
and Xiao-Dong Zhang\textsuperscript{1}\thanks{Email: \href{mailto:xiaodong@sjtu.edu.cn}{xiaodong@sjtu.edu.cn}}
\\[0.4em]
\textsuperscript{1}School of Mathematical Sciences, Shanghai Jiao Tong University, Shanghai, China\\
\textsuperscript{2}Department of Mathematics, The Hong Kong University of Science and Technology, \\Clear Water Bay, Kowloon, Hong Kong
}

\date{\today}
\maketitle
\begin{abstract}
Motivated by Alon and Frankl's work on Tur\'an-type problems for a matching and a clique, we investigate the maximum number of edges in $r$-uniform Berge-triangle-free hypergraphs with matching number at most $s$. For $r=3$ and $r=4$, we determine the exact value of this extremal number, thereby extending Gy\H{o}ri's classical theorem on Berge-$K_3$-free hypergraphs. For $r\ge 5$, we provide a linear lower bound and an upper bound. In addition, we derive an improved upper bound by means of a reduction lemma, which asserts that every such hypergraph admits a graph whose edge set can be partitioned into induced matchings. This recovers the result of Gr\'osz, Methuku, and Tompkins that $\operatorname{ex}_r\bigl(n,\mathcal B(K_3)\bigr)=o(n^2)$ for every fixed $r\geq5$.


\textbf{AMS Classification (2020)}: 05C35, 05C65.
\end{abstract}

\section{Introduction}

For a given integer $r\geq2$, an $r$-uniform hypergraph, or an $r$-graph, $\mathcal H$ consists of a vertex set $V(\mathcal H)$ and an edge set $E(\mathcal H)\subseteq\binom{V(\mathcal H)}{r}$. We write $v(\mathcal H):=|V(\mathcal H)|$ and $e(\mathcal H):=|E(\mathcal H)|$.

A matching in a hypergraph is a set of pairwise disjoint edges. The matching number $\nu(\mathcal H)$ of a hypergraph $\mathcal H$ is the maximum size of a matching in $\mathcal H$. For a matching $M$, write $V(M):=\bigcup_{m\in M}m$. For positive integers $r$ and $t$, let $M_t^r$ denote the $r$-graph consisting of $t$ pairwise disjoint edges; when $r=2$, write $M_t:=M_t^2$.

Let $\mathcal F$ be a family of $r$-graphs. We say that an $r$-graph $\mathcal H$ is $\mathcal F$-free if it does not contain any member of $\mathcal F$ as a subhypergraph. We denote by $\operatorname{ex}_r(n,\mathcal F)$ the Tur\'an number of $\mathcal F$, that is, the maximum number of edges in an $n$-vertex $\mathcal F$-free $r$-graph. In particular, when $r=2$, we write $\operatorname{ex}_2(n,\mathcal F)$ simply as $\operatorname{ex}(n,\mathcal F)$. For a given graph $F$, determining $\operatorname{ex}(n,F)$, a problem initiated by Tur\'an~\cite{T41}, is one of the central topics in extremal combinatorics. For positive integers $\ell\geq2$ and $s\geq1$, we denote by $K_\ell$ the complete graph on $\ell$ vertices and by $M_s$ a matching of size $s$. Let $T(n,\ell)$ be the $\ell$-partite graph on $n$ vertices with the maximum number of edges, and write $t(n,\ell)=e(T(n,\ell))$. Tur\'an's theorem~\cite{T41} states that $\operatorname{ex}(n,K_{\ell+1})=t(n,\ell)$, and the Erd\H{o}s--Stone--Simonovits theorem~\cite{ErdosStone1946,ErdosSimonovits1966} states that every graph $F$ with chromatic number $\chi(F)\geq2$ satisfies $\operatorname{ex}(n,F)=(1-1/(\chi(F)-1)+o(1))n^2/2$. The Erd\H{o}s--Gallai theorem~\cite{ErdosGallai59} determines the maximum number of edges in an $n$-vertex graph with no matching of size $s+1$. We refer the reader to the survey by F\"uredi and Simonovits~\cite{FurediSimonovits2013} for further results and background. Alon and Frankl~\cite{AlonFrankl24JCTB_K_M} determined the Tur\'an problem for $\mathcal F=\{K_{\ell+1},M_{s+1}\}$. Let $G(n,\ell,s)$ denote the complete $\ell$-partite graph on $n$ vertices in which one part has size $n-s$ and the remaining $\ell-1$ parts induce a copy of $T(s,\ell-1)$. We write $g(n,\ell,s):=e(G(n,\ell,s))$.

\begin{theorem}[Alon and Frankl~\cite{AlonFrankl24JCTB_K_M}]\label{AlonFrankl}
For $n\geq2s+1$ and $\ell\geq2$,
\begin{equation*}
\operatorname{ex}(n,\{K_{\ell+1},M_{s+1}\})=\max\{t(2s+1,\ell),g(n,\ell,s)\}.
\end{equation*}
\end{theorem}
Building on this result, the Tur\'an problem with bounded matching number has been extensively studied by many researchers; we refer the reader to~\cite{G2024_Turan_bounded_matching,xue2024_Turan_bounded_matching,zhao_and_Lu2024_Turan_bounded_matching,zhu2025_Turan_bounded_matching} for a collection of related results and further developments.

For $r$-uniform hypergraphs, the Tur\'an problem is considerably more difficult. Few exact Tur\'an numbers are known, and the exact Tur\'an number of the hyper-clique remains unknown to date. Among the few known results, the Tur\'an problem of a matching was proved by Frankl~\cite{FRANKL20131068}.

\begin{theorem}[Frankl~\cite{FRANKL20131068}]
\label{thm:frankl-matching}
Let $r,s\geq1$ and $n\geq(2s+1)r-s$. Then
\begin{equation*}
\operatorname{ex}_r\left(n,M_{s+1}^r\right)=\sum_{i=1}^s\binom{s}{i}\binom{n-s}{r-i}.
\end{equation*}
Equality holds only for families isomorphic to $\mathcal A(n,r,s)$, where $\mathcal A(n,r,s)$ is the $r$-graph in which every edge intersects a fixed $s$-set.
\end{theorem}

On the other hand, Gerbner and Palmer~\cite{Gerbner_and_Palmer_Def_of_BG_first} generalized earlier definitions of Berge paths and Berge cycles and introduced the following definition of a Berge-$G$.
\begin{defn}\label{def:berge-g}
Let $G$ be a graph. A hypergraph $\mathcal H$ is a Berge-$G$ if there exists a bijection $\varphi\colon E(G)\to E(\mathcal H)$ such that $e\subseteq\varphi(e)$ for every $e\in E(G)$.
\end{defn}
For a graph $G$, let $\mathcal B(G)$ denote the family of all $r$-uniform Berge-$G$ hypergraphs, where $r$ is the ambient uniformity. Tur\'an problems on Berge hypergraphs have been extensively studied, yielding numerous extremal results. In particular, Gerbner, Methuku, and Palmer~\cite{gerbner2020general} developed a general framework of lemmas for Berge-Tur\'an hypergraph problems. For $r\in\{3,4\}$, Gy\H{o}ri~\cite{gyHori2006triangle} proved that every $n$-vertex Berge-$K_3$-free $r$-graph contains at most $n^2/(8(r-2))$ edges for sufficiently large $n$. For every fixed $r\geq5$, Gr\'osz, Methuku, and Tompkins~\cite{grosz2020uniformity_threshold} proved $\operatorname{ex}_r(n,\mathcal B(K_3))=o(n^2)$ as $n\to\infty$.
\begin{theorem}[Gy\H{o}ri~\cite{gyHori2006triangle}; Gr\'osz, Methuku, and Tompkins~\cite{grosz2020uniformity_threshold}]
(1) If $r\in\{3,4\}$, then, for all sufficiently large $n$,
\begin{equation*}
\operatorname{ex}_r(n,\mathcal B(K_3))\leq\frac{n^2}{8(r-2)}.
\end{equation*}

(2) For every fixed $r\geq5$, as $n\to\infty$,
\begin{equation*}
\operatorname{ex}_r(n,\mathcal B(K_3))=o(n^2).
\end{equation*}
\end{theorem}

In addition, Gy\H{o}ri and Lemons~\cite{gyHori2012Berge_cycle_given_length} showed that if an $r$-uniform hypergraph on $n$ vertices is Berge-$C_{2k}$-free and $r\geq3$, then it has at most $O(n^{1+1/k})$ edges, matching the order of magnitude in the graph case (cf.\ the even cycle theorem of Bondy and Simonovits~\cite{bondyS_1974_even_cycles}). They proved the same upper bound for Berge-$C_{2k+1}$-free $r$-uniform hypergraphs. F\"uredi, Kostochka, and Luo~\cite{furedi2019_long_Berge_cycle} showed that for $n\geq k\geq r+3$, every $n$-vertex $r$-uniform hypergraph containing no Berge cycles of length at least $k$ has at most $\frac{n-1}{k-2}\binom{k-1}{r}$ edges, and this bound is attained whenever $k-2\mid n-1$. The same authors~\cite{FKL2020_Berge_cycles_non_uniform} extended the study of long Berge cycles to non-uniform hypergraphs. Ergemlidze, Gy\H{o}ri, Methuku, Salia, Tompkins, and Zamora~\cite{ergemlidze2020avoiding} determined the exact maximum number of edges when $k=r+2$ or $k=r+1$. Gy\H{o}ri, Lemons, Salia, and Zamora~\cite{GLSZ2021_structure_long_Berge_cycles} studied $r$-uniform hypergraphs with no Berge cycles of length at least $k$ for $k\leq r$.

Berge cycles are also related to bipartite Tur\'an problems. Let an $r$-uniform hypergraph $\mathcal H$ be Berge-$C_k$-free, and let $G$ be its incidence bipartite graph with parts $V(\mathcal H)$ and $E(\mathcal H)$. A cycle of length $2k$ in $G$ would give a Berge cycle of length $k$ in $\mathcal H$, so $G$ is $C_{2k}$-free; it has $r\cdot|E(\mathcal H)|$ edges. We refer the reader to~\cite{naor2005_bipartite_turan,gyHori2006triangle,verstraete2016extremalC2k_survey} for further details. Other studied families include Berge-$K_4$-free~\cite{gyarfas2019Berge_K4,zhu2020_Berge_K4_turan} and triangle-free~\cite{frankl2024triangle} $3$-uniform hypergraphs, Berge-book-free hypergraphs~\cite{gerbner2024Berge_book}, Berge forests~\cite{zhou2025Berge_forest}, Berge-path-free linear $3$-uniform hypergraphs~\cite{gyHori2025linear}, connected hypergraphs without long Berge paths~\cite{GSZ2021_connected_Berge_paths}, and disjoint unions of Berge paths~\cite{zhan2025tur}.

Moreover, the hypergraph Tur\'an problem with bounded matching number, motivated by the analogous graph-theoretic framework of Alon and Frankl, has been the subject of extensive research. For a comprehensive overview of results prior to 2011, we refer the reader to the survey of Keevash~\cite{keevash2011hypergraph}. Gerbner, Tompkins, and Zhou~\cite{GTZ2025} investigated the analogous problems for $r$-uniform hypergraphs and obtained exact results for certain families. Wang, Wang, and Yang~\cite{WWY2025} proved that every $F_5$-free $3$-graph $H$ with matching number at most $s$ has at most $s\lfloor(n-s)^2/4\rfloor$ edges for $n\geq30(s+1)$ and $s\geq3$, where the generalized triangle $F_5$ is the $3$-graph on vertex set $\{a,b,c,d,e\}$ with edge set $\{abc,abd,cde\}$. For the $3$-graph $F_{3,2}$ with edge set $\{123,145,245,345\}$, Chen, Liu, Qi, and Yang~\cite{CLQY25} determined the exact value of $\operatorname{ex}_3(n,\{F_{3,2},M_{s+1}^3\})$ for every positive integer $s$ and all $n\geq12s^2$. Yang, Zeng, and Zhang~\cite{YangZengZhang2025} obtained a hypergraph analogue of the Alon--Frankl theorem.

For $n,m\in\mathbb N$ with $m\leq n$, we denote $[n]:=\{1,2,\ldots,n\}$ and $[m,n]:=\{m,m+1,\ldots,n\}$.

\begin{defn}\label{def:extremal-constructions}
We define the hypergraphs used below in the indicated parameter ranges.

\begin{itemize}
\item \textbf{Case $r=3$ and $n\geq3s$:}
\begin{align*}
V(\mathcal H_3(s,n))&=\{v_1,\ldots,v_{2s},u_1,\ldots,u_{n-2s}\},\\
E(\mathcal H_3(s,n))&=\bigl\{\{v_{2i-1},v_{2i},u_j\}:i\in[s],\ j\in[n-2s]\bigr\}.
\end{align*}

\item \textbf{Case $r\geq4$, $s=1$, and $n\geq r$:}
Write $n-1=rk+\ell$ with $k,\ell\in\mathbb Z_{\geq0}$ and $0\leq\ell<r$. Put
\begin{equation*}
\mathcal P=\bigl\{a_1b_{ri+1}\cdots b_{ri+r-1},\ a_1b_{ri+2}\cdots b_{r(i+1)}:0\leq i<k\bigr\}.
\end{equation*}
The vertex set is $V(\mathcal H_r(1,n))=\{a_1,b_1,\ldots,b_{rk+\ell}\}$, and
\begin{equation*}
E(\mathcal H_r(1,n))=
\begin{cases}
\mathcal P,&r\nmid n,\\[4pt]
\mathcal P\cup\{a_1b_{rk+1}\cdots b_{rk+r-1}\},&r\mid n.
\end{cases}
\end{equation*}

\item \textbf{Case $r=4$, $s\geq2$, and $n\geq4s$:}
The vertex set is $V(\mathcal H_4(s,n))=\{v_1,\ldots,v_{2s},u_1,\ldots,u_{n-2s}\}$, and
\begin{equation*}
E(\mathcal H_4(s,n))=\biggl\{v_{2i-1}v_{2i}u_{2j-1}u_{2j}:i\in[s],\ j\in\biggl[\biggl\lfloor\frac{n-2s}{2}\biggr\rfloor\biggr]\biggr\}.
\end{equation*}
\end{itemize}
\end{defn}

 \begin{theorem}\label{main_big_theorem}
Let $r\in\{3,4\}$ and $s\geq1$, and let $\mathcal H$ be an $n$-vertex Berge-$K_3$-free $r$-graph with $\nu(\mathcal H)=s$. Then $n\geq rs$ and $e(\mathcal H)\leq e(\mathcal H_r(s,n))$. Equality is attained by $\mathcal H_r(s,n)$. In particular, if $r=3$, then
\begin{equation*}
e(\mathcal H)\leq s(n-2s)=e(\mathcal H_3(s,n)).
\end{equation*}
For all positive integers $n$ and $s$, the corresponding extremal number with matching number at most $s$ is
\begin{equation*}
\operatorname{ex}_3\Bigl(n,\mathcal B(K_3)\cup\{M_{s+1}^3\}\Bigr)=
\begin{cases}
\left\lfloor\dfrac{n^2}{8}\right\rfloor,&n\leq4s,\\[6pt]
s(n-2s),&n>4s.
\end{cases}
\end{equation*}
For $r=4$, the corresponding extremal number is
\begin{equation*}
\operatorname{ex}_4\bigl(n,\mathcal B(K_3)\cup\{M_{s+1}^{4}\}\bigr)=
\begin{cases}
2,&n=5,\\
\left\lfloor\dfrac{n}{4}\right\rfloor\left(\left\lfloor\dfrac{n}{2}\right\rfloor-\left\lfloor\dfrac{n}{4}\right\rfloor\right),&n\leq4s\text{ and }n\neq5,\\[10pt]
\dfrac{n-1}{2},&s=1,\ n>5,\text{ and }4\mid(n-1),\\[6pt]
s\left\lfloor\dfrac{n-2s}{2}\right\rfloor,&\text{otherwise.}
\end{cases}
\end{equation*}
\end{theorem}

For $r\geq5$, the upper bound below holds for every matching number. The lower bound holds for fixed $r$ and $s$ when $n$ is sufficiently large; Theorem~\ref{r_ge_5_lower_construction} gives a range in which $s$ may depend on $n$.

\begin{theorem}\label{main_big_theorem_2}
For every fixed pair of integers $r\geq5$ and $s\geq1$, there is an integer $n_0=n_0(r,s)$ such that, for every $n\geq n_0$,
\begin{equation*}
\frac{sn}{2048r(r-1)}\leq\operatorname{ex}_r\Bigl(n,\mathcal B(K_3)\cup\{M_{s+1}^r\}\Bigr)\leq s+\frac{2sn}{r}.
\end{equation*}
Moreover, for every $n>r$,
\begin{equation*}
\operatorname{ex}_r\Bigl(n,\mathcal B(K_3)\cup\{M_2^r\}\Bigr)=e(\mathcal H_r(1,n)).
\end{equation*}
$\mathcal H_r(1,n)$ attains this value.
\end{theorem}

   For every fixed $r\geq5$, the unrestricted extremal number is $o(n^2)$. Hence a lower bound of order $sn$ cannot hold uniformly when $s$ is proportional to $n$.

 A matching in a graph is induced if its edges span no additional edges. For positive integers $n$ and $t$, let $\operatorname{RM}(n,t)$ be the maximum number of edges in an $n$-vertex graph whose edge set can be partitioned into at most $t$ induced matchings.

 The following theorem bounds $\operatorname{ex}_r\bigl(n,\mathcal B(K_3)\cup\{M_{s+1}^r\}\bigr)$ in terms of $\operatorname{RM}(n,rs)$.
\begin{theorem}\label{RS_reduction_theorem}
Let $r\geq5$. Let $n$ and $s$ be positive integers such that $rs\leq n$. Then
\begin{equation*}
\operatorname{ex}_r\Bigl(n,\mathcal B(K_3)\cup\{M_{s+1}^r\}\Bigr)\leq\operatorname{RM}(n,rs)+\left(1+\left\lfloor\frac{r}{3}\right\rfloor\right)s.
\end{equation*}
\end{theorem}
For fixed $r\geq5$ and $n\geq r$, set $s=\lfloor n/r\rfloor$. Every $n$-vertex $r$-graph has matching number at most $s$, so $\operatorname{ex}_r(n,\mathcal B(K_3))=\operatorname{ex}_r\bigl(n,\mathcal B(K_3)\cup\{M_{s+1}^r\}\bigr)$. Since $rs\leq n$, Theorem~\ref{RS_reduction_theorem} bounds this common value by $\operatorname{RM}(n,n)+(1+\lfloor r/3\rfloor)s$. The Ruzsa--Szemer\'edi theorem gives $\operatorname{RM}(n,n)=o(n^2)$; see, for example,~\cite{fox_huang_sudakov_RS}. The second term is at most $(1+\lfloor r/3\rfloor)n/r=o(n^2)$. Hence $\operatorname{ex}_r(n,\mathcal B(K_3))=o(n^2)$.

 Section 2 contains the notation and preliminary lemmas. Section 3 proves Theorem~\ref{main_big_theorem}, and Section 4 proves Theorems~\ref{main_big_theorem_2} and~\ref{RS_reduction_theorem}.

\section{Preliminaries}
For a hypergraph $\mathcal H$, a vertex $v\in V(\mathcal H)$, and a set $W\subseteq V(\mathcal H)$, let $d_{\mathcal H}(v)$ be the number of edges of $\mathcal H$ containing $v$, and let $\mathcal H[W]$ be the subhypergraph induced by $W$.

Let $r\geq3$, let $\mathcal H$ be a Berge-$K_3$-free $r$-graph, and fix $e\in E(\mathcal H)$. Define $X_{\mathcal H}(e)=\{f\in E(\mathcal H)\setminus\{e\}:f\cap e\neq\varnothing\}$ and $\overline X_{\mathcal H}(e)=X_{\mathcal H}(e)\cup\{e\}$. For $i\in\{1,2\}$, let $X_{\mathcal H}^{i}(e)=\{f\in X_{\mathcal H}(e):|f\cap e|=i\}$, and let $X_{\mathcal H}^{3}(e)=\{f\in X_{\mathcal H}(e):|f\cap e|\geq3\}$. Thus $X_{\mathcal H}(e)=X_{\mathcal H}^{1}(e)\sqcup X_{\mathcal H}^{2}(e)\sqcup X_{\mathcal H}^{3}(e)$, where $\sqcup$ denotes disjoint union.

For every $f\in X_{\mathcal H}(e)$, put $B_f=f\cap e$ and $A_f=f\setminus e$. Both sets are nonempty and $|A_f|+|B_f|=r$.

\begin{claim}\label{claim:outside-intersection}
If $f,g\in X_{\mathcal H}(e)$ are distinct and $A_f\cap A_g\neq\varnothing$, then $B_f=B_g=\{b\}$ for some $b\in e$.
\end{claim}

\begin{proof}
Choose $z\in A_f\cap A_g$. If $B_f$ and $B_g$ are not the same singleton, then there are distinct vertices $x\in B_f$ and $y\in B_g$. The vertices $x,y,z$ are distinct, and the distinct hyperedges $e,f,g$ contain the pairs $xy,xz,yz$, respectively. They form a Berge triangle, contrary to the hypothesis.
\end{proof}

The map $f\mapsto A_f$ is injective on $X_{\mathcal H}(e)$: if $A_f=A_g$, then Claim~\ref{claim:outside-intersection} gives $B_f=B_g$, and hence $f=g$.

\begin{claim}\label{claim:three-outside-parts}
Let $f,g,h\in X_{\mathcal H}(e)$ be pairwise distinct. If $A_f\cap A_g\neq\varnothing$ and $A_f\cap A_h\neq\varnothing$, then there are vertices $b\in e$ and $z\notin e$ such that $B_f=B_g=B_h=\{b\}$ and $A_f\cap A_g=A_f\cap A_h=A_g\cap A_h=\{z\}$.
\end{claim}

\begin{proof}
Claim~\ref{claim:outside-intersection} gives $B_f=B_g=B_h=\{b\}$ for some $b\in e$. Put $I=A_f\cap A_g$ and $J=A_f\cap A_h$. If $I\neq J$, then, after interchanging $g$ and $h$ if necessary, choose $x\in I\setminus J$ and $y\in J$; the hyperedges $f,g,h$ contain the pairs $xy,bx,by$, respectively, and form a Berge triangle. Hence $I=J$. If this common set contains distinct vertices $x$ and $y$, then $f,g,h$ again contain $xy,bx,by$, respectively, so $I=J=\{z\}$ for some $z$. The set $A_g\cap A_h$ contains $z$; if it also contained $w\neq z$, then $g,f,h$ would contain $zw,bz,bw$, respectively, and would form a Berge triangle. Thus $A_g\cap A_h=\{z\}$.
\end{proof}

\begin{claim}\label{claim:two-large-outside-parts}
If distinct $f,g\in X_{\mathcal H}(e)$ satisfy $|A_f\cap A_g|\geq2$, then $A_h\cap(A_f\cup A_g)=\varnothing$ for every $h\in X_{\mathcal H}(e)\setminus\{f,g\}$.
\end{claim}

\begin{proof}
If $A_h\cap A_f\neq\varnothing$, Claim~\ref{claim:three-outside-parts} would give $|A_f\cap A_g|=1$, a contradiction. Interchanging $f$ and $g$ excludes $A_h\cap A_g\neq\varnothing$.
\end{proof}

For every nonempty proper subset $B\subsetneq e$, put $\operatorname{Lk}_{\mathcal H}(B,e)=\{A_f:f\in X_{\mathcal H}(e),\ B_f=B\}$. Let $\mathcal V_1(e)=\{b\in e:\operatorname{Lk}_{\mathcal H}(\{b\},e)\neq\varnothing\}$. Let $\mathcal B_2(e)=\{B\subsetneq e:|B|=2,\ \operatorname{Lk}_{\mathcal H}(B,e)\neq\varnothing\}$ and $\mathcal B_3(e)=\{B\subsetneq e:|B|\geq3,\ \operatorname{Lk}_{\mathcal H}(B,e)\neq\varnothing\}$.

\begin{claim}\label{claim:B-sets-and-matching}
Any two distinct members of $\mathcal B_2(e)\cup\mathcal B_3(e)$ are disjoint, and $|\operatorname{Lk}_{\mathcal H}(B,e)|=1$ for every $B\in\mathcal B_3(e)$. Moreover, every subfamily $\mathcal C\subseteq\mathcal B_2(e)\cup\mathcal B_3(e)$ satisfies $|\mathcal C|+|\mathcal V_1(e)\setminus\bigcup_{B\in\mathcal C}B|\leq\nu(\mathcal H)$.
\end{claim}

\begin{proof}
Let $B,B'\in\mathcal B_2(e)\cup\mathcal B_3(e)$ be distinct, and choose $f,g\in X_{\mathcal H}(e)$ with $B_f=B$ and $B_g=B'$. If $B\cap B'\neq\varnothing$, then there are distinct $x,y,z\in e$ such that $x,z\in B$ and $y,z\in B'$. The hyperedges $e,f,g$ contain $xy,xz,yz$, respectively, and form a Berge triangle. Hence $B\cap B'=\varnothing$.

Fix $B\in\mathcal B_3(e)$. If distinct $f,g\in X_{\mathcal H}(e)$ satisfied $B_f=B_g=B$, then any three vertices of $B$, together with $e,f,g$, would form a Berge triangle. Thus $|\operatorname{Lk}_{\mathcal H}(B,e)|=1$.

For each $B\in\mathcal C$, choose $f_B\in X_{\mathcal H}(e)$ with $B_{f_B}=B$. For each $b\in\mathcal V_1(e)\setminus\bigcup_{B\in\mathcal C}B$, choose $f_b\in X_{\mathcal H}(e)$ with $B_{f_b}=\{b\}$. The intersections of the chosen edges with $e$ are pairwise disjoint. If two chosen edges met outside $e$, Claim~\ref{claim:outside-intersection} would force their intersections with $e$ to be the same singleton, which is excluded by the choices. The chosen edges form a matching, and their number gives the asserted inequality.
\end{proof}

Define $L_{\mathcal H}(e)$ by $V(L_{\mathcal H}(e))=\bigcup_{f\in X_{\mathcal H}(e)}A_f$ and $E(L_{\mathcal H}(e))=\{A_f:f\in X_{\mathcal H}(e)\}$. The injectivity of $f\mapsto A_f$ gives $e(L_{\mathcal H}(e))=|X_{\mathcal H}(e)|$, while $V(L_{\mathcal H}(e))\subseteq V(\mathcal H)\setminus e$ gives $v(L_{\mathcal H}(e))\leq v(\mathcal H)-r$.

Two edges of a hypergraph lie in the same component if they can be joined by a sequence of edges in which consecutive edges intersect. For integers $k\geq2$ and $p\geq1$, let $S_p^{(k)}$ be the $k$-graph with edges $\{z\}\cup P_1,\ldots,\{z\}\cup P_p$, where the sets $P_1,\ldots,P_p$ are pairwise disjoint $(k-1)$-sets and $z\notin\bigcup_iP_i$. For $2\leq q\leq k-1$, let $T_{2,q}^{(k)}$ be the $k$-graph consisting of two distinct edges whose intersection has cardinality $q$.

\begin{lemma}\label{lem:link-components}
Every component of $L_{\mathcal H}(e)$ is a single hyperedge of size between $1$ and $r-1$, a copy of $S_p^{(r-1)}$ for some $p\geq2$, or a copy of $T_{2,q}^{(r-1)}$ for some $2\leq q\leq r-2$. If $|B_f|\geq2$, then $A_f$ forms a one-edge component of $L_{\mathcal H}(e)$; every component containing at least two edges consists of sets $A_f$ for which $B_f=\{b\}$ with one fixed $b\in e$; and the one-edge components whose edges have size at most $r-3$ are in bijection with $\mathcal B_3(e)$.
\end{lemma}

\begin{proof}
If $|B_f|\geq2$, Claim~\ref{claim:outside-intersection} gives $A_f\cap A_g=\varnothing$ for every $g\neq f$, so $A_f$ forms a one-edge component. Claim~\ref{claim:B-sets-and-matching} shows that every $B\in\mathcal B_3(e)$ gives exactly one such edge, of size $r-|B|\leq r-3$, and every one-edge component whose edge has size at most $r-3$ corresponds to a member of $\mathcal B_3(e)$.

Let $C$ be a component containing at least two edges. The preceding paragraph and Claim~\ref{claim:outside-intersection} show that there is a vertex $b\in e$ such that $B_f=\{b\}$ for every $A_f\in E(C)$. If two edges $A_f,A_g\in E(C)$ satisfy $|A_f\cap A_g|\geq2$, Claim~\ref{claim:two-large-outside-parts} shows that no third edge meets either of them, so $C$ is a copy of $T_{2,q}^{(r-1)}$ with $q=|A_f\cap A_g|$. Since $A_f$ and $A_g$ are distinct $(r-1)$-sets, $2\leq q\leq r-2$.

It remains to consider a component $C$ in which every two intersecting edges meet in one vertex. Choose intersecting edges $A_{f_0}$ and $A_{f_1}$ and write $A_{f_0}\cap A_{f_1}=\{z\}$. If $A_g$ is any other edge that meets $A_{f_0}$, Claim~\ref{claim:three-outside-parts}, applied to $f_0,f_1,g$, gives $A_{f_0}\cap A_g=\{z\}$. Now take a sequence $A_{g_0},A_{g_1},\ldots,A_{g_t}$ of distinct edges in $C$ with $g_0=f_0$ and with consecutive edges intersecting. Its first intersection is $\{z\}$ by the preceding sentence, and successive applications of Claim~\ref{claim:three-outside-parts} to $g_{i-1},g_i,g_{i+1}$ give $A_{g_i}\cap A_{g_{i+1}}=\{z\}$. Hence every edge in $C$ contains $z$, so every two edges in $C$ intersect exactly in $\{z\}$. Their complements of $z$ are pairwise disjoint, and $C$ is a copy of $S_p^{(r-1)}$ with $p=e(C)\geq2$.
\end{proof}

\begin{cor}\label{cor:edge-neighborhood-bound}
Let $r\geq5$, let $\mathcal H$ be an $n$-vertex Berge-$K_3$-free $r$-graph, and let $e\in E(\mathcal H)$ and $m=v(L_{\mathcal H}(e))$. Then
\begin{equation*}
|X_{\mathcal H}(e)|\leq\frac{2(m+r)}{r}\leq\frac{2n}{r}.
\end{equation*}
\end{cor}

\begin{proof}
Let $\mathcal E_3=\{A_f:f\in X_{\mathcal H}(e),\ |B_f|\geq3\}$, put $q=|\mathcal E_3|$, and let $u=|\bigcup_{A\in\mathcal E_3}A|$. Claim~\ref{claim:B-sets-and-matching} and Lemma~\ref{lem:link-components} give $q=|\mathcal B_3(e)|$, $u=qr-\sum_{B\in\mathcal B_3(e)}|B|$, and $\sum_{B\in\mathcal B_3(e)}|B|\leq r$.

Every component $C$ of $L_{\mathcal H}(e)$ that does not consist of an edge in $\mathcal E_3$ satisfies $e(C)\leq2v(C)/r$. The edge in a one-edge component of this kind has size $r-2$ or $r-1$; a component $S_p^{(r-1)}$ has $p$ edges and $1+p(r-2)$ vertices; and a component $T_{2,q'}^{(r-1)}$ has two edges and $2(r-1)-q'\geq r$ vertices. In the $S_p^{(r-1)}$ case, $2(1+p(r-2))-pr=2+p(r-4)\geq0$, and the other cases follow from the displayed vertex counts.

The one-edge components whose unique edges lie in $\mathcal E_3$ use $u$ vertices. Consequently,
\begin{align*}
|X_{\mathcal H}(e)|&=e(L_{\mathcal H}(e))\\
&\leq q+\frac{2(m-u)}{r}\\
&=\frac{2m}{r}-q+\frac{2}{r}\sum_{B\in\mathcal B_3(e)}|B|\\
&\leq\frac{2m}{r}+2=\frac{2(m+r)}{r}.
\end{align*}
Finally, $m\leq n-r$, and hence $2(m+r)/r\leq2n/r$.
\end{proof}
\section{Proof of Theorem~\ref{main_big_theorem}} \label{sec:proof-main}

\subsection{The case \texorpdfstring{$r=3$}{r=3}} \label{subsec:r3}

\begin{lemma}\label{lem:r3-outside-matching}
Let $\mathcal H$ be an $n$-vertex Berge-$K_3$-free $3$-graph, and let $M=\{m_1,\ldots,m_s\}$ be a maximum matching in $\mathcal H$. If $E(\mathcal H[V(M)])=M$, then
\begin{equation*}
e(\mathcal H)\leq s+s(n-3s).
\end{equation*}
If, in addition, $n=3s+2$ and $V(\mathcal H)\setminus V(M)=\{x,y\}$, then
\begin{equation*}
e(\mathcal H)+\bigl|\{h\in E(\mathcal H):\{x,y\}\subseteq h\}\bigr|-s\leq2s.
\end{equation*}
\end{lemma}

\begin{proof}
Set $W:=V(\mathcal H)\setminus V(M)$. Since $M$ is maximum, every edge in $E(\mathcal H)\setminus M$ meets $V(M)$. The assumption $E(\mathcal H[V(M)])=M$ therefore implies that $1\leq|h\cap V(M)|\leq2$ for every $h\in E(\mathcal H)\setminus M$.

We first show that for every $h\in E(\mathcal H)\setminus M$ there is a pair $(w,m)\in W\times M$ such that $w\in h$, $h\cap m\neq\varnothing$, and $h$ is the only edge of $\mathcal H$ that contains $w$ and meets $m$.

Suppose that $h=\{u,x,y\}$, where $u\in m\in M$ and $x,y\in W$. If neither $(x,m)$ nor $(y,m)$ has this property, choose edges $f_x\neq h$ and $f_y\neq h$ such that $x\in f_x$, $f_x\cap m\neq\varnothing$, $y\in f_y$, and $f_y\cap m\neq\varnothing$. If $f_x$ met $m$ in a vertex $a\neq u$, then $h,f_x,m$ would represent the pairs $xu,xa,ua$, respectively, and form a Berge triangle. Thus $f_x\cap m=\{u\}$, and the same argument gives $f_y\cap m=\{u\}$. The edges $f_x$ and $f_y$ are distinct, since otherwise their common $3$-set would be $\{u,x,y\}=h$. Hence $h,f_x,f_y$ represent the pairs $xy,xu,yu$, respectively, which is again a Berge triangle. Therefore at least one of $(x,m)$ and $(y,m)$ has this property.

Suppose that $h=\{u,v,w\}$, where $u,v\in m\in M$ and $w\in W$. If $(w,m)$ did not have this property, choose $g\neq h$ with $w\in g$ and $g\cap m\neq\varnothing$, and choose $a\in g\cap m$. If $a\in\{u,v\}$, let $b$ be the other vertex of $\{u,v\}$; otherwise, let $b=u$. The edges $g,m,h$ then represent the pairs $wa,ab,bw$, respectively, and form a Berge triangle. Thus $(w,m)$ has this property.

It remains to consider $h=\{u,v,w\}$ with $u\in m\in M$, $v\in m'\in M$, $m\neq m'$, and $w\in W$. If neither $(w,m)$ nor $(w,m')$ had this property, choose $f,g\neq h$ such that $w\in f$, $f\cap m\neq\varnothing$, $w\in g$, and $g\cap m'\neq\varnothing$. The argument in the first case gives $f\cap m=\{u\}$ and $g\cap m'=\{v\}$. The edges $f$ and $g$ are distinct, since otherwise their common $3$-set would equal $h$. The edges $h,f,g$ then represent the pairs $uv,uw,vw$, respectively, and form a Berge triangle.

For every edge in $E(\mathcal H)\setminus M$, choose one such pair. Two distinct edges cannot receive the same pair. Since there are $s|W|=s(n-3s)$ pairs in $W\times M$, we obtain

\begin{equation*}
e(\mathcal H)-s\leq s(n-3s),
\end{equation*}

which proves the first assertion.

Now suppose that $W=\{x,y\}$. Every edge containing $\{x,y\}$ has the form $h=\{x,y,u\}$ with $u\in m\in M$, and at least one of $(x,m)$ and $(y,m)$ has the property established above. Such an edge received one of these two pairs in the preceding assignment; assign the other pair to it as an additional pair. This additional pair was not used earlier. If it has the stated property, then $h$ is its unique edge and $h$ previously received the other pair; if it does not, then it was not eligible for the preceding assignment.

For each $m\in M$, at most one edge containing $\{x,y\}$ has its third vertex in $m$. Otherwise, two such edges with third vertices $u,v\in m$ would, together with $m$, represent the pairs $xu,xv,uv$ of a Berge triangle. Thus the additional pairs are distinct, and all assigned pairs belong to the $2s$-element set $\{x,y\}\times M$. It follows that

\begin{equation*}
e(\mathcal H)-s+\bigl|\{h\in E(\mathcal H):\{x,y\}\subseteq h\}\bigr|\leq2s,
\end{equation*}

which proves the second assertion.
\end{proof}

\begin{theorem}\label{thm:r3-exact}
Let $s\geq1$ and $n\geq3s$. If $\mathcal H$ is an $n$-vertex Berge-$K_3$-free $3$-graph with $\nu(\mathcal H)=s$, then
\begin{equation*}
e(\mathcal H)\leq s(n-2s).
\end{equation*}
The hypergraph $\mathcal H_3(s,n)$ attains equality.
\end{theorem}

\begin{proof}
Let $M=\{h_1,\ldots,h_s\}$ be a maximum matching in $\mathcal H$, and let $\mathcal H_M:=\mathcal H[V(M)]$. Define the spanning subhypergraph $\mathcal K$ of $\mathcal H$ by
\begin{equation*}
E(\mathcal K):=M\cup\bigl(E(\mathcal H)\setminus E(\mathcal H_M)\bigr).
\end{equation*}
The hypergraph $\mathcal K$ is Berge-$K_3$-free, contains $M$, and is a subhypergraph of $\mathcal H$. Hence $\nu(\mathcal K)=s$, and $E(\mathcal K[V(M)])=M$. Lemma~\ref{lem:r3-outside-matching} gives $e(\mathcal K)\leq s+s(n-3s)$. Since $e(\mathcal K)=s+e(\mathcal H)-e(\mathcal H_M)$, we obtain
\begin{equation}\label{eq:r3-outside-bound}
e(\mathcal H)-e(\mathcal H_M)\leq s(n-3s).
\end{equation}

We next prove
\begin{equation}\label{eq:r3-inside-bound}
e(\mathcal H_M)\leq s^2.
\end{equation}
The assertion holds for $s=1$ because $\mathcal H_M$ is a simple $3$-graph on three vertices and therefore has at most one edge. Assume that $s\geq2$, fix $i\in[s]$, write $h_i=\{x_1,x_2,x_3\}$, and set $M_i:=M\setminus\{h_i\}$. For $j\in[3]$ and $1\leq j<k\leq3$, define
\begin{align*}
\mathcal E_j&:=\{f\in E(\mathcal H_M)\setminus\{h_i\}:f\cap h_i=\{x_j\}\},\\
\mathcal E_{jk}&:=\{f\in E(\mathcal H_M)\setminus\{h_i\}:f\cap h_i=\{x_j,x_k\}\}.
\end{align*}

At most one of $\mathcal E_{12},\mathcal E_{13},\mathcal E_{23}$ is nonempty. For example, if $f\in\mathcal E_{12}$ and $g\in\mathcal E_{13}$, then $f,g,h_i$ represent the pairs $x_1x_2,x_1x_3,x_2x_3$, respectively. The other two choices of nonempty families give the same contradiction after permuting $x_1,x_2,x_3$. Relabel the vertices of $h_i$ so that $\mathcal E_{13}=\mathcal E_{23}=\varnothing$.

Let $\mathcal K_{12}$ be the $3$-graph defined by
\begin{align*}
V(\mathcal K_{12})&:=V(M_i)\cup\{x_1,x_2\},\\
E(\mathcal K_{12})&:=M_i\cup\mathcal E_1\cup\mathcal E_2\cup\mathcal E_{12}.
\end{align*}
The hypergraph $\mathcal K_{12}$ is a subhypergraph of $\mathcal H_M$ and is therefore Berge-$K_3$-free. It has $3(s-1)+2$ vertices and contains the matching $M_i$ of size $s-1$. Its order excludes a matching of size $s$, so $\nu(\mathcal K_{12})=s-1$. Moreover, $E(\mathcal K_{12}[V(M_i)])=M_i$. The second assertion of Lemma~\ref{lem:r3-outside-matching}, with $s-1$ in place of $s$, gives
\begin{equation*}
|\mathcal E_1|+|\mathcal E_2|+2|\mathcal E_{12}|\leq2(s-1).
\end{equation*}

Let $\mathcal K_3$ be the $3$-graph with vertex set $V(M_i)\cup\{x_3\}$ and edge set $M_i\cup\mathcal E_3$. It is a subhypergraph of $\mathcal H_M$ and is therefore Berge-$K_3$-free. It has $3(s-1)+1$ vertices, contains $M_i$, and has no matching of size $s$ by its order. Therefore $\nu(\mathcal K_3)=s-1$, while $E(\mathcal K_3[V(M_i)])=M_i$. The first assertion of Lemma~\ref{lem:r3-outside-matching} gives $e(\mathcal K_3)\leq2(s-1)$, and hence
\begin{equation*}
|\mathcal E_3|\leq s-1.
\end{equation*}

The preceding two inequalities imply
\begin{align*}
\sum_{x\in h_i}d_{\mathcal H_M}(x)
&=3+|\mathcal E_1|+|\mathcal E_2|+|\mathcal E_3|+2|\mathcal E_{12}|\\
&\leq3s.
\end{align*}
This estimate holds for every $i\in[s]$. Since the edges of $M$ partition $V(M)$, summing over $i$ gives
\begin{equation*}
3e(\mathcal H_M)=\sum_{i=1}^s\sum_{x\in h_i}d_{\mathcal H_M}(x)\leq3s^2,
\end{equation*}
which proves \eqref{eq:r3-inside-bound}. Combining \eqref{eq:r3-outside-bound} and \eqref{eq:r3-inside-bound} yields
\begin{equation*}
e(\mathcal H)\leq s(n-3s)+s^2=s(n-2s).
\end{equation*}

It remains to verify the construction. For $i\in[s]$, set $A_i:=\{v_{2i-1},v_{2i}\}$. Definition~\ref{def:extremal-constructions} gives
\begin{equation*}
E(\mathcal H_3(s,n))=\{A_i\cup\{u_j\}:i\in[s],\ j\in[n-2s]\}.
\end{equation*}
Therefore $e(\mathcal H_3(s,n))=s(n-2s)$. Every edge contains one of the pairwise disjoint sets $A_i$, so a matching uses each $A_i$ at most once and has size at most $s$. Since $n-2s\geq s$, the edges $A_i\cup\{u_i\}$ for $i\in[s]$ form a matching. Hence $\nu(\mathcal H_3(s,n))=s$.

Suppose that $\mathcal H_3(s,n)$ contains a Berge triangle. Every pair contained in an edge is either $A_i$ for some $i$, or consists of one vertex of some $A_i$ and one vertex $u_j$. It follows that the three core vertices must be the two vertices of some $A_i$ and one vertex $u_j$. The two pairs joining $u_j$ to the vertices of $A_i$ are both contained only in the edge $A_i\cup\{u_j\}$, so they cannot have distinct representative edges. Thus $\mathcal H_3(s,n)$ is Berge-$K_3$-free and attains the bound.
\end{proof}

\begin{cor}\label{cor:r3-bounded-matching}
For all positive integers $n$ and $s$,

\begin{equation*}
\operatorname{ex}_3\Bigl(n,\mathcal B(K_3)\cup\{M_{s+1}^3\}\Bigr)=
\begin{cases}
\left\lfloor\dfrac{n^2}{8}\right\rfloor,&n\leq4s,\\[6pt]
s(n-2s),&n>4s.
\end{cases}
\end{equation*}

\end{cor}

\begin{proof}
Let $\mathcal H$ be an $n$-vertex Berge-$K_3$-free $3$-graph with $\nu(\mathcal H)\leq s$, and put $t:=\nu(\mathcal H)$. Then $0\leq t\leq\min\{s,\lfloor n/3\rfloor\}$. If $t=0$, then $E(\mathcal H)=\varnothing$; if $t\geq1$, Theorem~\ref{thm:r3-exact} gives $e(\mathcal H)\leq t(n-2t)$. Conversely, $\mathcal H_3(t,n)$ has matching number $t$ and $t(n-2t)$ edges whenever $1\leq t\leq\min\{s,\lfloor n/3\rfloor\}$, while the empty hypergraph gives the value for $t=0$. Consequently,

\begin{equation}\label{eq:r3-layer-maximum}
\operatorname{ex}_3\Bigl(n,\mathcal B(K_3)\cup\{M_{s+1}^3\}\Bigr)=\max_{0\leq t\leq\min\{s,\lfloor n/3\rfloor\}}t(n-2t).
\end{equation}

Set $f(t):=t(n-2t)$. Since $f(t+1)-f(t)=n-4t-2$, if $n>4s$, then $s\leq\lfloor n/3\rfloor$ and $f$ is increasing on the integers from $0$ to $s$. Thus the maximum in \eqref{eq:r3-layer-maximum} is $f(s)=s(n-2s)$.

Suppose that $n\leq4s$. For $n\leq2$, the only possible matching number is zero and both sides of the asserted formula are zero. Let $n\geq3$ and put $t_0:=\lfloor(n+1)/4\rfloor$. The inequality $n\leq4s$ gives $t_0\leq s$, and $t_0\leq(n+1)/4\leq n/3$ gives $t_0\leq\lfloor n/3\rfloor$. The sign of $f(t+1)-f(t)$ shows that $t_0$ is an integer maximizer of $f$. If $n=4q+r$ with $r\in\{0,1,2,3\}$, then $t_0=q$ for $r\in\{0,1,2\}$ and $t_0=q+1$ for $r=3$. Substitution gives $f(t_0)=2q^2,2q^2+q,2q^2+2q,2q^2+3q+1$, respectively; these four values equal $\lfloor n^2/8\rfloor$. This proves the first case.
\end{proof}

\begin{cor}[Gy\H{o}ri~\cite{gyHori2006triangle}]\label{cor:r3-turan}
For every positive integer $n$,

\begin{equation*}
\operatorname{ex}_3\bigl(n,\mathcal B(K_3)\bigr)=\left\lfloor\frac{n^2}{8}\right\rfloor.
\end{equation*}

\end{cor}

\begin{proof}
For $n\leq2$, both sides are zero. Let $n\geq3$ and set $s:=\lfloor n/3\rfloor$. No $n$-vertex $3$-graph contains $M_{s+1}^3$, so Corollary~\ref{cor:r3-bounded-matching} applies without an additional restriction. If $n=5$, its second case gives $s(n-2s)=3=\lfloor n^2/8\rfloor$. If $n=3,4$ or $n\geq6$, then $n\leq4s$, and its first case gives the asserted value.
\end{proof}

 \subsection{The case \texorpdfstring{$r=4$}{r=4}}\label{subsec:r4}
In order to prove the extremal results for $r=4$, we first establish a local estimate around a fixed edge.

\begin{lemma}\label{lem:r4-local}
Let $\mathcal H$ be a Berge-$K_3$-free $4$-graph and let $e\in E(\mathcal H)$. Then
\begin{equation*}
|\overline X_{\mathcal H}(e)|\leq
\begin{cases}
\dfrac{v(L_{\mathcal H}(e))+3}{2},&v(L_{\mathcal H}(e))\equiv1\pmod 4,\\[6pt]
\left\lfloor\dfrac{v(L_{\mathcal H}(e))+2}{2}\right\rfloor,&v(L_{\mathcal H}(e))\not\equiv1\pmod 4
\end{cases}.
\end{equation*}
\end{lemma}

\begin{proof}
For each $f\in X_{\mathcal H}(e)$, set $A_f=f\setminus e$ and $B_f=f\cap e$. Let $f,g\in X_{\mathcal H}(e)$ be distinct, and suppose that at least one of $|B_f|$ and $|B_g|$ is at least two. If $A_f\cap A_g\neq\varnothing$, choose $z\in A_f\cap A_g$ and distinct vertices $x\in B_f$ and $y\in B_g$; such a choice is possible because at least one of $B_f$ and $B_g$ has at least two vertices. The edges $f,g,e$ then contain the pairs $xz,yz,xy$, respectively, and form a Berge-$K_3$, a contradiction. Hence $A_f\cap A_g=\varnothing$. The same argument shows that if $|B_f|=|B_g|=1$ and $A_f\cap A_g\neq\varnothing$, then $B_f=B_g$.

Suppose that $B_f$ and $B_g$ are distinct, have size at least two, and intersect. When neither set contains the other, choose $x\in B_f\cap B_g$, $y\in B_f\setminus B_g$, and $z\in B_g\setminus B_f$; when $B_f\subsetneq B_g$, choose distinct $x,y\in B_f$ and $z\in B_g\setminus B_f$, and interchange $f$ and $g$ in the remaining case. In each case, $f,g,e$ contain the pairs $xy,xz,yz$, respectively, and form a Berge-$K_3$, a contradiction. Thus distinct sets among the $B_f$ of size at least two are disjoint.

There is at most one edge in $X_{\mathcal H}^{3}(e)$. Two distinct $3$-subsets of the $4$-set $e$ intersect, which is excluded by the preceding paragraph, while two distinct hyperedges having the same $3$-vertex intersection with $e$, together with $e$, form a Berge-$K_3$ on those three vertices. It follows also that $X_{\mathcal H}^{2}(e)=\varnothing$ whenever $X_{\mathcal H}^{3}(e)\neq\varnothing$, and that the sets $A_f$ with $f\in X_{\mathcal H}^{2}(e)$ are pairwise disjoint.

Fix $x\in e$ and consider the triples $A_f$ for which $B_f=\{x\}$. Triples corresponding to different choices of $x$ have disjoint vertex sets. Suppose that two such triples $A_f$ and $A_g$ intersect in two vertices. If a third such triple $A_h$ meets $A_f\cup A_g$, interchange $f$ and $g$ if necessary and choose $y\in A_h\cap A_f$. Choose $z\in A_f\cap A_g$ with $z\neq y$ when $y\in A_f\cap A_g$, and choose any $z\in A_f\cap A_g$ otherwise. The hyperedges $f,g,h$ contain the pairs $yz,xz,xy$, respectively, and form a Berge-$K_3$, a contradiction. Hence $A_f$ and $A_g$ form a component isomorphic to $T_{2,2}^{(3)}$.

Consider a component in which every two triples intersect in at most one vertex. Let $A_f,A_g,A_h$ be three consecutive triples on a path, with $A_f\cap A_g=\{u\}$ and $A_g\cap A_h=\{w\}$. If $u\neq w$, then $f,g,h$ contain the pairs $xu,uw,xw$, respectively, and form a Berge-$K_3$, a contradiction. Therefore $u=w$. Repeating this argument along paths shows that all triples in the component share one vertex and are otherwise pairwise disjoint, so the component is isomorphic to $S_p^{(3)}$ for some positive integer $p$.

The preceding intersection properties give the following possibilities for each component of $L_{\mathcal H}(e)$.
\begin{equation*}
\begin{array}{c|c|c}
\text{component}&\text{number of vertices}&\text{contribution to }|X_{\mathcal H}(e)|\\ \hline
\quad\text{component consisting of one singleton edge}&1&1\\
\text{one edge of size }2&2&1\\
S_p^{(3)}&2p+1&p\\
T_{2,2}^{(3)}&4&2
\end{array}.
\end{equation*}
A component consisting of one singleton edge occurs exactly when $X_{\mathcal H}^{3}(e)\neq\varnothing$.

Write $m=v(L_{\mathcal H}(e))$. If $L_{\mathcal H}(e)$ has no component consisting of one singleton edge, each component contributes at most one edge per two vertices. Hence $|\overline X_{\mathcal H}(e)|\leq1+\lfloor m/2\rfloor=\lfloor(m+2)/2\rfloor$.

Suppose that $L_{\mathcal H}(e)$ has a component consisting of one singleton edge. It is unique, there are no components consisting of an edge of size two, and every remaining component is isomorphic to $S_p^{(3)}$ or $T_{2,2}^{(3)}$. Let $q$ be the number of components isomorphic to some $S_p^{(3)}$, let their parameters be $p_1,\ldots,p_q$, and let $t$ be the number of components isomorphic to $T_{2,2}^{(3)}$. Then
\begin{equation}\label{eq:r4-component-count}
\begin{aligned}
m&=1+\sum_{i=1}^{q}(2p_i+1)+4t,\\
|\overline X_{\mathcal H}(e)|&=2+\sum_{i=1}^{q}p_i+2t=\frac{m+3-q}{2}.
\end{aligned}
\end{equation}
If $m$ is even, then $q$ is odd and at least one. If $m\equiv3\pmod 4$, then $q$ is even and cannot be zero, because $q=0$ would give $m=1+4t\equiv1\pmod 4$; hence $q\geq2$. If $m\equiv1\pmod 4$, then $q$ is even and may be zero. Equation~\eqref{eq:r4-component-count} now gives $|\overline X_{\mathcal H}(e)|\leq(m+2)/2$ when $m$ is even, $|\overline X_{\mathcal H}(e)|\leq(m+1)/2$ when $m\equiv3\pmod 4$, and $|\overline X_{\mathcal H}(e)|\leq(m+3)/2$ when $m\equiv1\pmod 4$.
\end{proof}

Since $v(L_{\mathcal H}(e))\leq n-4$ for an $n$-vertex $4$-graph, Lemma~\ref{lem:r4-local} gives
\begin{equation}\label{eq:r4-order-bound}
|\overline X_{\mathcal H}(e)|\leq
\begin{cases}
\dfrac{n-1}{2},&n\equiv1\pmod 4,\\[4pt]
\left\lfloor\dfrac{n-2}{2}\right\rfloor,&n\not\equiv1\pmod 4
\end{cases}.
\end{equation}
For $n\geq5$ and $n\equiv0,1,2,3\pmod 4$, respectively, the largest values from the first line of Lemma~\ref{lem:r4-local} are $(n-4)/2,(n-1)/2,(n-2)/2,(n-3)/2$, while the largest values from the second line are $(n-2)/2,(n-3)/2,(n-2)/2,(n-3)/2$. For $n=4$, the only possible value of $v(L_{\mathcal H}(e))$ is zero, and Lemma~\ref{lem:r4-local} gives $|\overline X_{\mathcal H}(e)|\leq1$.

The base case $s=1$ is handled by the following sharp bound.

\begin{lemma}\label{lem:r4-matching-one}
Let $\mathcal H$ be a Berge-$K_3$-free $4$-graph on $n\geq4$ vertices with $\nu(\mathcal H)\leq1$. Then
\begin{equation*}
e(\mathcal H)\leq
\begin{cases}
\dfrac{n-1}{2},&n\equiv1\pmod 4,\\[4pt]
\left\lfloor\dfrac{n-2}{2}\right\rfloor,&n\not\equiv1\pmod 4
\end{cases}.
\end{equation*}
The corresponding bound is attained by $\mathcal H_4(1,n)$.
\end{lemma}

\begin{proof}
If $E(\mathcal H)=\varnothing$, the inequality holds. Otherwise, fix $e\in E(\mathcal H)$. Since $\nu(\mathcal H)\leq1$, every edge meets $e$, so $E(\mathcal H)=\overline X_{\mathcal H}(e)$. The bound follows from \eqref{eq:r4-order-bound}.

Every edge of $\mathcal H_4(1,n)$ contains $a_1$. The two hyperedges arising from one index intersect outside $a_1$ in two vertices, whereas hyperedges arising from different indices are disjoint outside $a_1$; the possible final hyperedge when $4\mid n$ is also disjoint from the other hyperedges outside $a_1$. In a Berge-$K_3$, the three representing hyperedges have pairwise intersections containing three distinct vertices. Among any three hyperedges of $\mathcal H_4(1,n)$, either every pair intersects only in $a_1$, or two hyperedges have additional common vertices and the third meets each of them only in $a_1$. Neither pattern has the required three distinct intersection vertices, so $\mathcal H_4(1,n)$ is Berge-$K_3$-free. All its edges contain $a_1$, so its matching number is one. Write $n=4k+l+1$ with integers $k\geq0$ and $0\leq l<4$. The construction has $2k$ edges when $l\in\{0,1,2\}$ and $2k+1$ edges when $l=3$, which gives the two stated values.
\end{proof}

Having disposed of the initial case, we prepare two auxiliary lemmas needed for the induction step.

\begin{lemma}\label{lem:r4-improved-edge}
Let $\mathcal H$ be a Berge-$K_3$-free $4$-graph on $N\equiv1\pmod 4$ vertices. If $\nu(\mathcal H)\geq2$, then some $e\in E(\mathcal H)$ satisfies $|\overline X_{\mathcal H}(e)|\leq(N-3)/2$.
\end{lemma}

\begin{proof}
Suppose that every $a\in E(\mathcal H)$ satisfies $|\overline X_{\mathcal H}(a)|>(N-3)/2$. Since these cardinalities are integers, \eqref{eq:r4-order-bound} yields
\begin{equation}\label{eq:r4-all-local-max}
|\overline X_{\mathcal H}(a)|=\frac{N-1}{2}\qquad\text{for every }a\in E(\mathcal H).
\end{equation}

Fix $a\in E(\mathcal H)$ and write $m=v(L_{\mathcal H}(a))$. If $m<N-4$ and $m\equiv1\pmod 4$, then $m\leq N-8$, so Lemma~\ref{lem:r4-local} gives $|\overline X_{\mathcal H}(a)|\leq(N-5)/2$. If $m<N-4$ and $m\not\equiv1\pmod 4$, the same lemma gives $|\overline X_{\mathcal H}(a)|\leq(N-3)/2$. Both estimates contradict \eqref{eq:r4-all-local-max}. Therefore $V(L_{\mathcal H}(a))=V(\mathcal H)\setminus a$.

Now $m=N-4\equiv1\pmod 4$. If $X_{\mathcal H}^{3}(a)=\varnothing$, the proof of Lemma~\ref{lem:r4-local} gives $|\overline X_{\mathcal H}(a)|\leq(m+1)/2=(N-3)/2$, contrary to \eqref{eq:r4-all-local-max}. Hence $X_{\mathcal H}^{3}(a)\neq\varnothing$. Let $q$ denote the number of components of $L_{\mathcal H}(a)$ isomorphic to some $S_p^{(3)}$. The component count in the proof of Lemma~\ref{lem:r4-local} gives $|\overline X_{\mathcal H}(a)|=(m+3-q)/2$, so \eqref{eq:r4-all-local-max} gives $q=0$. Since $a$ was arbitrary, for every $c\in E(\mathcal H)$ there is exactly one edge, denoted by $c^*$, such that $|c\cap c^*|=3$, and every component of $L_{\mathcal H}(c)$ other than the component consisting of one singleton edge is isomorphic to $T_{2,2}^{(3)}$. Since $c$ meets $c^*$ in three vertices, the uniqueness of the corresponding edge for $c^*$ gives $(c^*)^*=c$.

Let $\{a,a^*\}\neq\{b,b^*\}$ and suppose that $(a\cup a^*)\cap(b\cup b^*)\neq\varnothing$. Choose $a_0\in\{a,a^*\}$ and $b_0\in\{b,b^*\}$ with $a_0\cap b_0\neq\varnothing$. Since the two sets of hyperedges are distinct and $(c^*)^*=c$ for every $c\in E(\mathcal H)$, they are disjoint, so $b_0\notin\{a_0,a_0^*\}$. The hyperedge $b_0$ therefore corresponds to an edge in a component of $L_{\mathcal H}(a_0)$ isomorphic to $T_{2,2}^{(3)}$. The other hyperedge corresponding to that component meets $b_0$ in three vertices and is consequently $b_0^*$. For some $z\in a_0$, we have $a_0\cap b_0=a_0\cap b_0^*=\{z\}$. Interchanging $a_0$ and $b_0$ gives $b_0\cap a_0=b_0\cap a_0^*=\{z\}$. Since $a_0\cap b_0^*=\{z\}$, applying the same argument with $b_0^*$ in place of $b_0$ gives $a_0^*\cap b_0^*=\{z\}$. Therefore
\begin{equation}\label{eq:r4-pair-intersection}
(a\cup a^*)\cap(b\cup b^*)=\{z\}\qquad\text{and}\qquad z\in a\cap a^*\cap b\cap b^*.
\end{equation}

Choose disjoint edges $a,b\in E(\mathcal H)$, which exist because $\nu(\mathcal H)\geq2$. Equation~\eqref{eq:r4-pair-intersection} implies that $a\cup a^*$ and $b\cup b^*$ are disjoint. Choose $t\in(b\cup b^*)\setminus(b\cap b^*)$. Since $V(L_{\mathcal H}(a))=V(\mathcal H)\setminus a$, some edge $h$ contains $t$ and meets $a$. The set $\{h,h^*\}$ differs from $\{a,a^*\}$ because $t\notin a\cup a^*$, and it differs from $\{b,b^*\}$ because neither $b$ nor $b^*$ meets $a$. Applying \eqref{eq:r4-pair-intersection} to $\{h,h^*\}$ and $\{b,b^*\}$ gives $t\in b\cap b^*$, contradicting the choice of $t$.
\end{proof}

\begin{lemma}\label{lem:r4-matching-two}
Let $\mathcal H$ be a Berge-$K_3$-free $4$-graph on $n\geq8$ vertices with $\nu(\mathcal H)\leq2$. Then $e(\mathcal H)\leq2\lfloor(n-4)/2\rfloor$.
\end{lemma}

\begin{proof}
Suppose first that $\nu(\mathcal H)\leq1$. When $n$ is even, Lemma~\ref{lem:r4-matching-one} gives $e(\mathcal H)\leq(n-2)/2\leq n-4=2\lfloor(n-4)/2\rfloor$. When $n$ is odd, we have $n\geq9$, and the same lemma gives $e(\mathcal H)\leq(n-1)/2\leq n-5=2\lfloor(n-4)/2\rfloor$. We may therefore assume that $\nu(\mathcal H)=2$.

Assume that $n\not\equiv1\pmod 4$. Fix $e\in E(\mathcal H)$ and let $\mathcal F=\mathcal H[V(\mathcal H)\setminus e]$. Then $\nu(\mathcal F)\leq1$ and $v(\mathcal F)=n-4\not\equiv1\pmod 4$. Equation~\eqref{eq:r4-order-bound} and Lemma~\ref{lem:r4-matching-one} give
\begin{equation*}
e(\mathcal H)=|\overline X_{\mathcal H}(e)|+e(\mathcal F)\leq\left\lfloor\frac{n-2}{2}\right\rfloor+\left\lfloor\frac{n-6}{2}\right\rfloor
=2\left\lfloor\frac{n-4}{2}\right\rfloor.
\end{equation*}

Assume that $n=4k+1$, where $k\geq2$. Choose $e$ according to Lemma~\ref{lem:r4-improved-edge} and set $\mathcal F=\mathcal H[V(\mathcal H)\setminus e]$. Then $|\overline X_{\mathcal H}(e)|\leq2k-1$ and Lemma~\ref{lem:r4-matching-one} gives $e(\mathcal F)\leq2k-2$. If $e(\mathcal F)\leq2k-3$, then $e(\mathcal H)\leq(2k-1)+(2k-3)=4k-4=2\lfloor(n-4)/2\rfloor$. It remains to consider $e(\mathcal F)=2k-2$.

Fix $a\in E(\mathcal F)$. Since $\nu(\mathcal F)\leq1$, every edge of $\mathcal F$ meets $a$, so $E(\mathcal F)=\overline X_{\mathcal F}(a)$. Write $m=v(L_{\mathcal F}(a))$. We have $m\leq4k-7$. Suppose that $m<4k-7$. If $m\equiv1\pmod 4$, then $m\leq4k-11$; if $m\not\equiv1\pmod 4$, then $m\leq4k-8$. Lemma~\ref{lem:r4-local} gives $e(\mathcal F)\leq2k-3$ in either case, a contradiction. Hence $V(L_{\mathcal F}(a))=V(\mathcal F)\setminus a$.

Now $m=4k-7\equiv1\pmod 4$ and $e(\mathcal F)=(m+3)/2$. If $X_{\mathcal F}^{3}(a)=\varnothing$, the proof of Lemma~\ref{lem:r4-local} gives $e(\mathcal F)\leq(m+1)/2$, a contradiction. Thus $X_{\mathcal F}^{3}(a)\neq\varnothing$. If $q$ denotes the number of components of $L_{\mathcal F}(a)$ isomorphic to some $S_p^{(3)}$, then the component count in the proof of Lemma~\ref{lem:r4-local} gives $e(\mathcal F)=(m+3-q)/2$, and hence $q=0$. There is exactly one edge $a^*$ with $|a\cap a^*|=3$, and every remaining component of $L_{\mathcal F}(a)$ is isomorphic to $T_{2,2}^{(3)}$.

All copies of $T_{2,2}^{(3)}$ in $L_{\mathcal F}(a)$ arise from hyperedges whose intersections with $a$ are the same vertex $x$. Otherwise, one hyperedge from each of two copies with different intersection vertices would be disjoint, contrary to $\nu(\mathcal F)\leq1$. If such a copy exists, then $x\in a\cap a^*$. Indeed, if $x\notin a\cap a^*$, the singleton set $a^*\setminus a$ is the edge of a one-edge component of $L_{\mathcal F}(a)$, and its vertex does not belong to a component isomorphic to $T_{2,2}^{(3)}$, so every hyperedge corresponding to that component is disjoint from $a^*$. If no such copy exists, choose any $x\in a\cap a^*$.

The hyperedges of $\mathcal F$ form $k-1$ pairs, including $\{a,a^*\}$, and the two hyperedges in each pair intersect in three vertices containing $x$. Hence there are pairwise disjoint sets $C_i=\{a_i,b_i,c_i,d_i\}$ and hyperedges $f_i,g_i$ for $i\in[k-1]$ such that
\begin{equation}\label{eq:r4-F-structure}
\begin{aligned}
V(\mathcal F)&=\{x\}\sqcup\bigsqcup_{i=1}^{k-1}C_i,\\
f_i&=\{x,a_i,b_i,c_i\},\qquad g_i=\{x,a_i,b_i,d_i\},\\
E(\mathcal F)&=\{f_i,g_i\mid i\in[k-1]\}.
\end{aligned}
\end{equation}
The sets in \eqref{eq:r4-F-structure} cover $V(\mathcal F)$ because $V(L_{\mathcal F}(a))=V(\mathcal F)\setminus a$, and different pairs are disjoint away from $x$ by the component decomposition in the proof of Lemma~\ref{lem:r4-local}.

Let $h\in X_{\mathcal H}(e)$. Suppose that $x\in h$ and $y\in h\cap C_i$. If $y\in\{a_i,b_i\}$, choose $z$ as the other vertex of $\{a_i,b_i\}$; otherwise, choose $z=a_i$. The hyperedge $h$ contains $xy$, and the two hyperedges $f_i,g_i$ can be assigned, in a suitable order, to the pairs $yz,xz$. This gives a Berge-$K_3$, so $h\setminus e=\{x\}$ whenever $x\in h$.

Suppose that $x\notin h$. The hyperedge $h$ cannot contain two vertices $y,z\in C_i$. Each of $a_i,b_i$ lies in both $f_i$ and $g_i$, while $c_i$ lies only in $f_i$ and $d_i$ lies only in $g_i$; therefore $f_i$ and $g_i$ can be assigned injectively to the pairs $xy$ and $xz$, while $h$ contains $yz$, and these three hyperedges form a Berge-$K_3$. The hyperedge $h$ also cannot contain vertices $y\in C_i$ and $z\in C_j$ with $i\neq j$, because one of $f_i,g_i$ contains $xy$, one of $f_j,g_j$ contains $xz$, and these two hyperedges together with $h$, which contains $yz$, form a Berge-$K_3$. It follows that $|h\setminus e|\leq1$. Since $h\neq e$, every $h\in X_{\mathcal H}(e)$ meets $e$ in exactly three vertices.

There cannot be two distinct edges $h,h'\in X_{\mathcal H}(e)$. If $h\cap e=h'\cap e$, then $e,h,h'$ form a Berge-$K_3$ on the common $3$-set. If $h\cap e\neq h'\cap e$, choose $u\in h\cap h'\cap e$, $v\in(h\cap e)\setminus h'$, and $w\in(h'\cap e)\setminus h$. The edges $h,h',e$ contain the pairs $uv,uw,vw$, respectively, and form a Berge-$K_3$. Thus $|X_{\mathcal H}(e)|\leq1$, and
\begin{equation*}
e(\mathcal H)=|\overline X_{\mathcal H}(e)|+e(\mathcal F)\leq2+(2k-2)=2k\leq4k-4=2\left\lfloor\frac{n-4}{2}\right\rfloor.
\end{equation*}
\end{proof}

Equipped with these results, we prove the following theorem by induction.

\begin{theorem}\label{thm:r4-main}
Let $s,n$ be positive integers with $n\geq4s$, and let $\mathcal H$ be an $n$-vertex Berge-$K_3$-free $4$-graph with $\nu(\mathcal H)\leq s$. Then
\begin{equation*}
e(\mathcal H)\leq
\begin{cases}
\dfrac{n-1}{2},&s=1\text{ and }n\equiv1\pmod 4,\\[6pt]
s\left\lfloor\dfrac{n-2s}{2}\right\rfloor,&\text{otherwise}
\end{cases}.
\end{equation*}
In both cases, $\mathcal H_4(s,n)$ attains the bound. Furthermore, for positive integers $n$ and $s$,
\begin{equation*}
\operatorname{ex}_4\Bigl(n,\mathcal B(K_3)\cup\{M_{s+1}^{4}\}\Bigr)=
\begin{cases}
2,&n=5,\\[4pt]
\left\lfloor\dfrac{n}{4}\right\rfloor\left(\left\lfloor\dfrac{n}{2}\right\rfloor-\left\lfloor\dfrac{n}{4}\right\rfloor\right),&n\leq4s\text{ and }n\neq5,\\[10pt]
\dfrac{n-1}{2},&s=1,\ n>5,\text{ and }4\mid(n-1),\\[6pt]
s\left\lfloor\dfrac{n-2s}{2}\right\rfloor,&\text{otherwise.}
\end{cases}.
\end{equation*}
\end{theorem}

\begin{proof}
The cases $s=1$ and $s=2$ follow from Lemmas~\ref{lem:r4-matching-one} and \ref{lem:r4-matching-two}. Let $s\geq3$ and assume that the result holds with $s-1$ in place of $s$.

Suppose that $\nu(\mathcal H)\leq s-1$. Since $s-1\geq2$ and $n\geq4s\geq4(s-1)$, the induction hypothesis applies with parameter $s-1$. Moreover, $n\geq4s$ gives $\lfloor(n-2s)/2\rfloor\geq s$, and
\begin{equation*}
e(\mathcal H)\leq(s-1)\left\lfloor\frac{n-2s+2}{2}\right\rfloor
\leq s\left\lfloor\frac{n-2s}{2}\right\rfloor.
\end{equation*}
We may therefore assume that $\nu(\mathcal H)=s$.

If $n\not\equiv1\pmod 4$, fix any $e\in E(\mathcal H)$. Equation~\eqref{eq:r4-order-bound} gives $|\overline X_{\mathcal H}(e)|\leq\lfloor(n-2)/2\rfloor$. If $n\equiv1\pmod 4$, choose $e$ according to Lemma~\ref{lem:r4-improved-edge}, which applies because $s\geq3$; that lemma gives $|\overline X_{\mathcal H}(e)|\leq(n-3)/2=\lfloor(n-2)/2\rfloor$.

Let $\mathcal H'=\mathcal H[V(\mathcal H)\setminus e]$. A matching of size $s$ in $\mathcal H'$ together with $e$ would give a matching of size $s+1$ in $\mathcal H$, so $\nu(\mathcal H')\leq s-1$. Since $v(\mathcal H')=n-4\geq4(s-1)$ and $s-1\geq2$, the induction hypothesis gives $e(\mathcal H')\leq(s-1)\lfloor(n-2s-2)/2\rfloor$.

If $n$ is even, the preceding two estimates give
\begin{equation*}
e(\mathcal H)\leq\frac{n-2}{2}+(s-1)\frac{n-2s-2}{2}=\frac{s(n-2s)}{2}.
\end{equation*}
If $n$ is odd, they give
\begin{equation*}
e(\mathcal H)\leq\frac{n-3}{2}+(s-1)\frac{n-2s-3}{2}=\frac{s(n-2s-1)}{2}.
\end{equation*}
These expressions equal $s\lfloor(n-2s)/2\rfloor$ for even and odd $n$, respectively.

Lemma~\ref{lem:r4-matching-one} verifies the construction for $s=1$. Let $s\geq2$, set $m=\lfloor(n-2s)/2\rfloor$, and define $A_i=\{v_{2i-1},v_{2i}\}$ for $i\in[s]$ and $B_j=\{u_{2j-1},u_{2j}\}$ for $j\in[m]$. The edges of $\mathcal H_4(s,n)$ are the sets $A_i\cup B_j$. Thus $e(\mathcal H_4(s,n))=sm$. Every matching uses distinct sets $A_i$, so it has at most $s$ edges. Since $n\geq4s$ implies $m\geq s$, the edges $A_i\cup B_i$ for $i\in[s]$ form a matching of size $s$.

Suppose that $\mathcal H_4(s,n)$ contains a Berge-$K_3$. Let $x,y,z$ be its three vertices, and let $h_{xy},h_{xz},h_{yz}$ be distinct hyperedges containing the indicated pairs. Each of $x,y,z$ belongs to a hyperedge, so none is the possible isolated vertex and each lies in some $A_i$ or $B_j$. If these vertices lie in three distinct sets, two of those sets are both of the form $A_i$ or both of the form $B_j$, whereas every hyperedge is the union of one $A_i$ and one $B_j$; the pair of vertices in those two sets is therefore contained in no hyperedge. If two of $x,y,z$, say $x$ and $y$, lie in the same set and $z$ lies in another, the existence of $h_{xz}$ forces one of these two sets to be of the form $A_i$ and the other to be of the form $B_j$. The hyperedges $h_{xz}$ and $h_{yz}$ are then both the union of these two sets, contradicting their distinctness. Three distinct vertices cannot lie in one two-element set. Hence $\mathcal H_4(s,n)$ is Berge-$K_3$-free and has $s\lfloor(n-2s)/2\rfloor$ edges.
We now turn to the proof of the second equality, namely the formula for $\mathrm{ex}_4\bigl(n,\mathcal B(K_3)\cup\{M_{s+1}^{4}\}\bigr)$. If $n\leq3$, both sides are zero. Suppose first that $4\leq n\leq4s$, and set $t=\lfloor n/4\rfloor$. Every matching in an $n$-vertex $4$-graph has at most $t$ edges, while $t\leq s$ and $n\geq4t$. For $n\neq5$, the first line of Theorem~\ref{thm:r4-main} does not apply with parameter $t$: if it did, then $t=1$ and $n\equiv1\pmod 4$, while $t=1$ gives $4\leq n<8$ and hence $n=5$. The theorem therefore gives $t\lfloor(n-2t)/2\rfloor=t(\lfloor n/2\rfloor-t)$, and $\mathcal H_4(t,n)$ attains this value. For $n=5$, Theorem~\ref{thm:r4-main} gives two edges, and $\mathcal H_4(1,5)$ consists of two $4$-edges meeting in three vertices. If $n>4s$, Theorem~\ref{thm:r4-main} with parameter $s$ gives the asserted value, and $\mathcal H_4(s,n)$ attains it.
\end{proof}

 Theorem~\ref{thm:r4-main} also determines the unrestricted extremal number.

\begin{cor}\label{cor:r4-turan}
For every positive integer $n$,
\begin{equation*}
\operatorname{ex}_4\bigl(n,\mathcal B(K_3)\bigr)=
\begin{cases}
2,&n=5,\\[4pt]
\left\lfloor\dfrac{n}{4}\right\rfloor\left(\left\lfloor\dfrac{n}{2}\right\rfloor-\left\lfloor\dfrac{n}{4}\right\rfloor\right),&n\neq5
\end{cases}.
\end{equation*}
\end{cor}

\begin{proof}
If $n\leq3$, both sides are zero. Let $n\geq4$ and set $t=\lfloor n/4\rfloor$. Every matching in an $n$-vertex $4$-graph has at most $t$ edges, and $n\geq4t$. Theorem~\ref{thm:r4-main} with parameter $t$ gives the stated upper bound, and $\mathcal H_4(t,n)$ attains it.
\end{proof}

\section{Bounds for \texorpdfstring{$r\geq5$}{r at least 5}}

For $r\in\{3,4\}$, the previous section determines the exact extremal numbers. For every fixed $r\geq5$, the relation $\operatorname{ex}_r\bigl(n,\mathcal B(K_3)\bigr)=o(n^2)$~\cite{grosz2020uniformity_threshold} shows that a lower bound $c_rsn$ with $c_r>0$ cannot hold uniformly when $s$ is proportional to $n$.

Throughout this section, $r\geq5$. We prove upper bounds for every matching number and a lower bound for the range stated in Theorem~\ref{r_ge_5_lower_construction}.
\subsection{A general linear upper bound}

Lemma~\ref{lem:link-components} also yields the exact bound when the matching number is one.

 \begin{lemma}\label{r_ge_5_s=1}
Let $n>r\geq5$, and let $\mathcal H$ be a Berge-$K_3$-free $r$-graph with $\nu(\mathcal H)=1$.

If $r\nmid n$, then $e(\mathcal H)\leq2\lfloor(n-1)/r\rfloor$.

If $r\mid n$, then $e(\mathcal H)\leq2n/r-1$.

Moreover, $\mathcal H_r(1,n)$ is a Berge-$K_3$-free $r$-graph with matching number one that attains equality.
\end{lemma}

   \begin{proof}
Fix $e\in E(\mathcal H)$. Since $\nu(\mathcal H)=1$, every edge other than $e$ meets $e$, and hence $e(\mathcal H)-1=|X_{\mathcal H}(e)|=e(L_{\mathcal H}(e))$. For a component $C$ of $L_{\mathcal H}(e)$, put $m_C=e(C)$ and $v_C=v(C)$. Distinct components have disjoint vertex sets, so $\sum_Cv_C\leq n-r$.

Lemma~\ref{lem:link-components} gives the following possibilities. If $C$ is a copy of $S_p^{(r-1)}$, then $m_C=p$ and $v_C=1+p(r-2)$, where a one-edge component whose edge has size $r-1$ may be treated as the case $p=1$. If $C$ is a copy of $T_{2,q}^{(r-1)}$, then $m_C=2$ and $v_C=2(r-1)-q\geq r$. A one-edge component corresponding to a member of $\mathcal B_2(e)$ has $m_C=1$ and $v_C=r-2$, while the possible one-edge component corresponding to a member of $\mathcal B_3(e)$ has $m_C=1$ and $v_C\geq1$. By Claim~\ref{claim:B-sets-and-matching} and $\nu(\mathcal H)=1$, the family $\mathcal B_2(e)\cup\mathcal B_3(e)$ has at most one member; in particular, if $\mathcal B_3(e)$ is nonempty, then $\mathcal B_2(e)$ is empty.

For every component $C$, put $\rho_C=v_C-r\lfloor m_C/2\rfloor$. If $C$ is a copy of $S_p^{(r-1)}$ with even $p$, or a copy of $T_{2,q}^{(r-1)}$, then $\rho_C\geq0$. If $C$ is a copy of $S_p^{(r-1)}$ with odd $p$, then $\rho_C\geq r-1$. A component corresponding to a member of $\mathcal B_2(e)$ has $\rho_C=r-2$, and the possible component corresponding to a member of $\mathcal B_3(e)$ has $\rho_C\geq1$.

Let $o$ be the number of components for which $m_C$ is odd. If $\mathcal B_3(e)$ is empty, pair the odd components arbitrarily and leave one unpaired when $o$ is odd. Each pair has total $\rho$ at least $2(r-2)\geq r$, and an unpaired component has $\rho\geq1$. If $\mathcal B_3(e)$ is nonempty and $o$ is odd, leave the corresponding component unpaired and pair the remaining components for which $m_C$ is odd. If $\mathcal B_3(e)$ is nonempty and $o$ is even, pair the corresponding component with one copy of $S_p^{(r-1)}$ for which $p$ is odd, and pair the remaining components for which $m_C$ is odd. In every case, the paired components have total $\rho$ at least $r$ per pair, and a possible unpaired component has $\rho\geq1$.

Put $m=|X_{\mathcal H}(e)|=\sum_Cm_C$. Since $m=2\sum_C\lfloor m_C/2\rfloor+o$, the preceding pairing gives
\begin{equation}\label{eq:matching-one-component-bound}
\sum_Cv_C\geq
\begin{cases}
rm/2,&m\text{ is even},\\[2pt]
r(m-1)/2+1,&m\text{ is odd}.
\end{cases}
\end{equation}
If $m=2t$, then \eqref{eq:matching-one-component-bound} gives $rt\leq n-r$, so $e(\mathcal H)=2t+1\leq2\lfloor n/r\rfloor-1$. If $m=2t+1$, then $rt+1\leq n-r$, so $e(\mathcal H)=2t+2\leq2\lfloor(n-1)/r\rfloor$. When $r\mid n$, the first bound is $2n/r-1$ and the second is smaller; when $r\nmid n$, their maximum is $2\lfloor(n-1)/r\rfloor$.

It remains to verify equality for the construction in Definition~\ref{def:extremal-constructions}. Write $n-1=rk+\ell$ with integers $k\geq1$ and $0\leq\ell<r$. Every edge of $\mathcal H_r(1,n)$ contains $a_1$, so its matching number is one. The edges other than the possible final edge occur in pairs indexed by $i\in\{0,\ldots,k-1\}$; the two edges in one pair intersect outside $a_1$, whereas edges from different pairs meet only at $a_1$, and the possible final edge also meets every preceding edge only at $a_1$. Among any three construction edges, at least two of the three pairwise intersections are contained in $\{a_1\}$, so the three edges cannot form a Berge triangle. If $r\nmid n$, then $e(\mathcal H_r(1,n))=2k=2\lfloor(n-1)/r\rfloor$; if $r\mid n$, then $e(\mathcal H_r(1,n))=2k+1=2n/r-1$. Thus the construction attains equality.
\end{proof}

Corollary~\ref{cor:edge-neighborhood-bound} gives the required general upper bound.

\begin{theorem}\label{r_ge_5_general_upper_bound}
Let $r\geq5$ and $s\geq1$ be integers, and let $\mathcal H$ be an $n$-vertex Berge-$K_3$-free $r$-graph with $\nu(\mathcal H)\leq s$. Then $e(\mathcal H)\leq s+2sn/r$. Consequently,
\begin{equation*}
\operatorname{ex}_r\Bigl(n,\mathcal B(K_3)\cup\{M_{s+1}^r\}\Bigr)\leq s+\frac{2sn}{r}.
\end{equation*}
\end{theorem}

\begin{proof}
   If $E(\mathcal H)=\varnothing$, the asserted bound holds. Assume that $E(\mathcal H)\neq\varnothing$.

Let $M=\{e_1,\ldots,e_\nu\}$ be a maximum matching in $\mathcal H$, where $1\leq\nu=\nu(\mathcal H)\leq s$. Maximality of $M$ implies that every edge in $E(\mathcal H)\setminus M$ meets at least one $e_i$. Thus $E(\mathcal H)\subseteq M\cup\bigcup_{i=1}^{\nu}X_{\mathcal H}(e_i)$.

 Corollary~\ref{cor:edge-neighborhood-bound} gives $|X_{\mathcal H}(e_i)|\leq2n/r$ for every $i\in[\nu]$. Therefore, $e(\mathcal H)\leq\nu+\sum_{i=1}^{\nu}|X_{\mathcal H}(e_i)|\leq\nu(1+2n/r)\leq s+2sn/r$. The extremal inequality follows because an $r$-graph is $M_{s+1}^r$-free exactly when its matching number is at most $s$.\end{proof}
\subsection{A lower-bound construction}
 Throughout this subsection, $\log$ denotes the natural logarithm, and an $r$-graph is linear if any two distinct edges intersect in at most one vertex.

  \begin{lemma}\label{constant_weight_code_lemma}
Let $d\geq128$ be even. There is a family $\mathcal C\subseteq\binom{[d]}{d/2}$ such that $|\mathcal C|\geq\exp(0.06d)$ and $|S\cap T|\leq3d/8$ for all distinct $S,T\in\mathcal C$.
\end{lemma}

\begin{proof}
Let $F$ be the graph with vertex set $\binom{[d]}{d/2}$ in which distinct sets $S,T$ are adjacent when $|S\cap T|>3d/8$. Put $N=|V(F)|$, let $\Delta$ be the maximum degree of $F$, and let $\alpha(F)$ be its independence number. Since the largest of the $d+1$ binomial coefficients of order $d$ is the central one, $N=\binom d{d/2}\geq2^d/(d+1)$.

Fix $S\in\binom{[d]}{d/2}$. Put $j=|S\setminus T|=d/2-|S\cap T|$. For a prescribed $j$, there are $\binom{d/2}{j}^2$ choices for $T$. Let $H_2(x)=-x\log_2x-(1-x)\log_2(1-x)$ for $0<x<1$. The estimate $\binom mk\leq2^{mH_2(k/m)}$ holds, and $H_2$ is increasing on $[0,1/2]$. Set $H_2(0)=H_2(1)=0$. Adjacency implies $j<d/8$, and these facts give
\begin{equation*}
\Delta+1\leq\sum_{j=0}^{\lfloor d/8\rfloor}\binom{d/2}{j}^2\leq\left(\frac{d}{8}+1\right)2^{H_2(1/4)d}.
\end{equation*}
The greedy independent-set bound gives
\begin{equation}\label{eq:constant-weight-greedy}
\alpha(F)\geq\frac{N}{\Delta+1}\geq\frac{2^{(1-H_2(1/4))d}}{(d+1)(d/8+1)}.
\end{equation}
We have $(1-H_2(1/4))\log2=3\log3/4-\log2>0.1308$. For real $d\geq128$, the function $d^{-1}\log((d+1)(d/8+1))$ is decreasing because the numerator determining the sign of its derivative is at most $2-\log(129\cdot17)<0$. Its value at $d=128$ is less than $0.061$. Taking logarithms in \eqref{eq:constant-weight-greedy} gives $\log\alpha(F)>(0.1308-0.061)d>0.06d$. An independent set of the asserted cardinality is therefore a suitable family $\mathcal C$.
\end{proof}

 \begin{theorem}\label{r_ge_5_lower_construction}
For every fixed $r\geq5$, there is an integer $n_0=n_0(r)$ such that, for every $n\geq n_0$ and every integer $s$ satisfying $1\leq s\leq n^{3/(25600\log\log n)}$, there is a linear $n$-vertex Berge-$K_3$-free $r$-graph $\mathcal H$ satisfying $\nu(\mathcal H)=s$ and $e(\mathcal H)\geq sn/(2048r(r-1))$. Consequently,
\begin{equation*}
\operatorname{ex}_r\Bigl(n,\mathcal B(K_3)\cup\{M_{s+1}^r\}\Bigr)\geq\frac{sn}{2048r(r-1)}.
\end{equation*}

 \end{theorem}

   \begin{proof}
Fix $r\geq5$. Choose $n_0(r)$ so that every $n\geq n_0(r)$ satisfies
\begin{equation}\label{eq:rge5-n0-conditions}
\begin{aligned}
\frac{\log n}{8192\log\log n}&\geq2,&\qquad \log\log n&\geq1,\\
\log(16r)&\leq\frac{1}{2}\log\log n,&\qquad \log(r-1)&\leq\frac{\log n}{256},\\
\frac{3}{25600\log\log n}&\leq\frac{1}{64},&\qquad 2n^{1/64}&\leq\frac{n}{2}.
\end{aligned}
\end{equation}
Such an integer exists because $r$ is fixed. Let $n\geq n_0(r)$ and let $s$ satisfy the hypothesis of the theorem. Define $\lambda=\log n/(8192\log\log n)$, and set

\begin{equation}\label{eq:rge5-construction-parameters}
d=64\lfloor\lambda\rfloor,\qquad p=16rd,\qquad \ell=\frac{d}{64}.
\end{equation}

Since $\lambda\geq2$, we have $\lfloor\lambda\rfloor\geq\lambda/2$. Thus $d$ is a multiple of $64$, $d\geq128$, and
\begin{equation}\label{eq:rge5-construction-d-bounds}
\frac{\log n}{256\log\log n}\leq d\leq\frac{\log n}{128\log\log n}.
\end{equation}
The lower bound on $d$ gives $s\leq n^{3/(25600\log\log n)}\leq\exp(0.06d)$. Lemma~\ref{constant_weight_code_lemma} therefore provides a family $\mathcal D\subseteq\binom{[d]}{d/2}$ of cardinality exactly $s$ such that $|S\cap T|\leq3d/8$ for all distinct $S,T\in\mathcal D$.

Write $\mathbb Z_p=\mathbb Z/p\mathbb Z$ for the cyclic additive group of residues modulo $p$.

For $S\in\mathcal D$, let $v_S\in\{0,1\}^d\subseteq\mathbb Z_p^d$ be the indicator vector of $S$, and define $\phi_S:\mathbb Z_p^d\to\mathbb Z_p$ by $\phi_S(x)=\sum_{i\in S}x_i$. Then $\phi_S(v_S)=d/2$, and, for distinct $S,T\in\mathcal D$, $\phi_S(v_T)=|S\cap T|\leq3d/8$.

The map $\phi_S$ is surjective because varying one coordinate indexed by an element of $S$ gives every residue modulo $p$. Let $I=\{0,1,\ldots,\ell-1\}\subseteq\mathbb Z_p$ and let $A_S=\phi_S^{-1}(I)$. Every fiber of $\phi_S$ has $p^{d-1}$ elements, so $|A_S|=\ell p^{d-1}$.

Set $B_0=(r-1)p^d$ and $k=\lfloor(n-s)/B_0\rfloor$. The upper bound on $d$ in \eqref{eq:rge5-construction-d-bounds} gives $d<\log n$. Together with $\log(16r)\leq\frac{1}{2}\log\log n$, this yields $\log(16rd)\leq3\log\log n/2$. Using also $\log(r-1)\leq\log n/256$, we obtain $\log B_0\leq\log n/256+3\log n/256=\log n/64$, and hence $B_0\leq n^{1/64}$. The fifth condition in \eqref{eq:rge5-n0-conditions} and the assumed upper bound on $s$ give $s\leq n^{1/64}$. Since $kB_0\geq n-s-B_0$, the final condition in \eqref{eq:rge5-n0-conditions} gives
\begin{equation}\label{eq:rge5-construction-block-count}
kB_0\geq n-2n^{1/64}\geq\frac{n}{2}.
\end{equation}
In particular, $k\geq1$.

Let $J=\{0,1,\ldots,r-2\}$, put $q=n-s-kB_0\geq0$, and let $W$ be a set of $q$ vertices disjoint from the other vertex sets below. Define

\begin{equation*}
V(\mathcal H)=\{c_S:S\in\mathcal D\}\mathbin{\sqcup}\{(b,j,z):b\in[k],\ j\in J,\ z\in\mathbb Z_p^d\}\mathbin{\sqcup}W.
\end{equation*}

For $b\in[k]$, $S\in\mathcal D$, and $x\in A_S$, define

\begin{equation}\label{eq:rge5-construction-edge}
e_b(S,x)=\{c_S\}\cup\{(b,j,x-jv_S):j\in J\},
\end{equation}

where the third coordinate is calculated in $\mathbb Z_p^d$, and let $E(\mathcal H)$ consist of all edges in \eqref{eq:rge5-construction-edge}. In $e_b(S,x)$, call $c_S$ the central vertex and $b$ the block index. The $r-1$ noncentral vertices have distinct second coordinates, so every edge has exactly $r$ vertices. An edge determines $S$ from its central vertex, $b$ from the first coordinate of its noncentral vertices, and $x$ from its vertex with second coordinate $0$; hence the parametrization is injective. The definitions of $q$ and $B_0$ give $v(\mathcal H)=s+k(r-1)p^d+q=n$, and the equality $|A_S|=\ell p^{d-1}$ gives $e(\mathcal H)=ks\ell p^{d-1}$.

Using \eqref{eq:rge5-construction-block-count}, $B_0=(r-1)p^d$, $\ell=d/64$, and $p=16rd$, we obtain

\begin{align}\label{eq:rge5-construction-edge-count}
e(\mathcal H)&=s(kB_0)\frac{\ell}{(r-1)p}\notag\\
&\geq\frac{sn}{2}\cdot\frac{1}{1024r(r-1)}=\frac{sn}{2048r(r-1)}.
\end{align}

Let $e_a(S,x)$ and $e_b(T,y)$ be distinct edges. If $S=T$, then both contain $c_S$, while a common noncentral vertex would force $a=b$, equality of its second coordinate in the two edges, and then $x=y$; hence their intersection is exactly $\{c_S\}$. Suppose that $S\neq T$. Their central vertices are distinct, and any common noncentral vertex forces $a=b$ and has the same second coordinate in both edges. If they had common vertices with distinct second coordinates $j,j'\in J$, then $(j-j')(v_S-v_T)=0$ in $\mathbb Z_p^d$. In a coordinate belonging to $S\mathbin{\triangle}T$, this equality gives $j-j'\equiv0\pmod p$, contrary to $0<|j-j'|\leq r-2<p$. Thus $\mathcal H$ is linear.

Call $S$ the color of an edge $e_b(S,x)$. Suppose that three edges of $\mathcal H$ form a Berge triangle, and call the three vertices of its underlying triangle its core vertices. The three edges cannot all have the same color, since any two distinct edges of one color intersect only in their common central vertex. Their colors are therefore pairwise distinct or consist of two occurrences of one color and one occurrence of another.

First suppose that the colors $S,T,U$ are pairwise distinct. Pairwise intersection forces the three block indices to be equal, so write the edges as $e_b(S,x)$, $e_b(T,y)$, and $e_b(U,z)$. Let the $S$--$T$, $T$--$U$, and $U$--$S$ intersection vertices have second coordinates $\alpha,\beta,\gamma\in J$, respectively. These values are pairwise distinct because each edge has only one vertex with a given second coordinate and the core vertices are distinct. The intersection equations imply
\begin{equation}\label{eq:rge5-construction-three-colors}
(\alpha-\gamma)v_S+(\beta-\alpha)v_T+(\gamma-\beta)v_U=0\qquad\text{in }\mathbb Z_p^d.
\end{equation}
Put $A=\alpha-\gamma$, $B=\beta-\alpha$, and $C=\gamma-\beta$. Then $A+B+C=0$, and $A,B,C$ are nonzero integers of absolute value at most $r-2$. Each coordinate of the left-hand side of \eqref{eq:rge5-construction-three-colors}, interpreted as an integer, has absolute value at most $3(r-2)<p$, so the congruence is an equality over the integers. Choose a coordinate on which $v_S$ and $v_T$ differ. If their entries there are $0$ and $1$, then the corresponding entry of $v_U$ is either $0$, which gives $B=0$, or $1$, which gives $B+C=-A=0$. If the entries of $v_S$ and $v_T$ are $1$ and $0$, the same argument with their roles interchanged gives $A=0$ or $B=0$. Each alternative contradicts the fact that $A,B,C$ are nonzero.

Now suppose that two edges have color $S$ and the third has color $T\neq S$. Since each of the two $S$-colored edges meets the $T$-colored edge, all three block indices are equal; write the edges as $e_b(S,x)$, $e_b(S,y)$, and $e_b(T,z)$. The two $S$-colored edges meet at $c_S$. Let their respective intersections with $e_b(T,z)$ have second coordinates $\alpha,\beta\in J$. These coordinates are distinct because the corresponding core vertices are distinct. We have $x-y=(\alpha-\beta)(v_S-v_T)$ in $\mathbb Z_p^d$. Let $a_x,a_y\in\{0,\ldots,\ell-1\}$ be the integer representatives of $\phi_S(x)$ and $\phi_S(y)$. Applying $\phi_S$ gives
\begin{equation}\label{eq:rge5-construction-two-colors}
a_x-a_y\equiv(\alpha-\beta)\left(\frac{d}{2}-|S\cap T|\right)\pmod p.
\end{equation}
Put $m=\alpha-\beta$ and $h=d/2-|S\cap T|$. The code property gives $d/8\leq h\leq d/2$, while $1\leq|m|\leq r-2$. Therefore $|a_x-a_y|\leq\ell-1<d/8\leq|mh|$ and $0<|mh-(a_x-a_y)|\leq(r-2)d/2+\ell-1<p$. This contradicts \eqref{eq:rge5-construction-two-colors}. The two possible color patterns are excluded, so $\mathcal H$ is Berge-$K_3$-free.

Every edge $e_b(S,x)$ contains $c_S$ and contains no $c_T$ for $T\neq S$, so a matching has at most one edge of each color and $\nu(\mathcal H)\leq s$.

Order $\mathcal D$ as $S_1,\ldots,S_s$. Since $k\geq1$, construct a matching using only block $b=1$. Suppose that pairwise disjoint edges $e_1(S_h,x_h)$ have been chosen for $h<i$. For fixed $h<i$, a candidate edge $e_1(S_i,x)$ meets $e_1(S_h,x_h)$ only if $x=x_h+j(v_{S_i}-v_{S_h})$ for some $j\in J$. Thus each previously chosen edge excludes at most $r-1$ values of $x$.

We have $s\leq\exp(0.06d)<2^d$. Equation~\eqref{eq:rge5-construction-parameters} gives

\begin{equation*}
|A_{S_i}|=\ell p^{d-1}=\frac{d(16rd)^{d-1}}{64}=2^d\frac{d}{128}(8rd)^{d-1}>(r-1)2^d>(r-1)(s-1).
\end{equation*}

At step $i$, fewer than $|A_{S_i}|$ values are excluded, so choose $x_i\in A_{S_i}$ for which $e_1(S_i,x_i)$ is disjoint from all previously selected edges. Continuing through all members of $\mathcal D$ produces a matching of size $s$. Hence $\nu(\mathcal H)=s$.

The constructed hypergraph is linear, Berge-$K_3$-free, has $n$ vertices and matching number $s$, and satisfies \eqref{eq:rge5-construction-edge-count}. It is therefore $\mathcal B(K_3)\cup\{M_{s+1}^r\}$-free, which proves the theorem.
\end{proof}

For fixed integers $r\geq5$ and $s\geq1$, the function $n^{3/(25600\log\log n)}$ tends to infinity. Hence there is an integer $N=N(r,s)$ such that Theorem~\ref{r_ge_5_lower_construction} applies for every $n\geq N$. Together with Theorem~\ref{r_ge_5_general_upper_bound}, this proves the two numerical inequalities asserted in Theorem~\ref{main_big_theorem_2}; Lemma~\ref{r_ge_5_s=1}
  proves its matching-one assertion.
 \subsection{A reduction to induced matchings}

   For a hypergraph $\mathcal H$, a pair $p\in\binom{V(\mathcal H)}{2}$ is private to an edge $h\in E(\mathcal H)$ if $h$ is the unique edge of $\mathcal H$ containing $p$.

\begin{lemma}\label{private_pair_core_claim}
Let $\mathcal H$ be a Berge-$K_3$-free $r$-graph, let $h\in E(\mathcal H)$, and let $U\subseteq h$ with $|U|\geq3$. If $U$ contains no pair private to $h$, then there is a unique edge $\tau(h,U)\in E(\mathcal H)\setminus\{h\}$ such that $U\subseteq\tau(h,U)$. Moreover, the edges containing any pair in $\binom{U}{2}$ are precisely $h$ and $\tau(h,U)$.
\end{lemma}

\begin{proof}
For each $xy\in\binom{U}{2}$, choose an edge $h_{xy}\neq h$ containing $xy$. Let $x,y,z$ be distinct vertices of $U$. If $h_{xy},h_{yz},h_{zx}$ are pairwise distinct, they form a Berge triangle on $x,y,z$. If exactly two of them are distinct, those two edges together with $h$ form a Berge triangle on $x,y,z$. Hence $h_{xy}=h_{yz}=h_{zx}$. The graph whose vertices are the pairs in $\binom{U}{2}$, with two pairs adjacent when their union has three vertices, is connected, so all chosen edges are equal to one edge $\tau(h,U)$ and $U\subseteq\tau(h,U)$. If a pair in $\binom{U}{2}$ were contained in a third edge, choosing that edge for this pair and repeating the argument would force it to equal $\tau(h,U)$.
 \end{proof}

 \begin{theorem}\label{RS_reduction_lemma}
Let $r\geq5$, let $n,s$ be positive integers, and let $\mathcal H$ be an $n$-vertex Berge-$K_3$-free $r$-graph with $\nu(\mathcal H)\leq s$. There is a simple graph $G$ on $V(\mathcal H)$ such that $e(G)\geq e(\mathcal H)-(1+\lfloor r/3\rfloor)s$ and $E(G)$ can be partitioned into at most $rs$ induced matchings. Consequently,

\begin{equation*}
\operatorname{ex}_r\Bigl(n,\mathcal B(K_3)\cup\{M_{s+1}^r\}\Bigr)\leq\operatorname{RM}(n,rs)+\left(1+\left\lfloor\frac{r}{3}\right\rfloor\right)s.
\end{equation*}

\end{theorem}

 \begin{proof}

If $E(\mathcal H)=\varnothing$, take $G$ to be the empty graph on $V(\mathcal H)$. Assume that $E(\mathcal H)\neq\varnothing$. Let $M=\{e_1,\ldots,e_\nu\}$ be a maximum matching in $\mathcal H$, where $\nu\leq s$. Every edge $h\in E(\mathcal H)\setminus M$ meets at least one edge of $M$. Since $M$ is maximum, $\nu=\nu(\mathcal H)$. Choose an index $\iota(h)\in[\nu]$ for which $|h\cap e_{\iota(h)}|$ is maximum, say that $h$ is assigned to $\iota(h)$, and put $B_h=h\cap e_{\iota(h)}$ and $A_h=h\setminus e_{\iota(h)}$.

For a fixed $i\in[\nu]$, there are at most $\lfloor r/3\rfloor$ edges $h$ assigned to $i$ with $|B_h|\geq3$. If two such edges $h,g$ have $B_h=B_g$, then any three vertices of this common set, together with $e_i,h,g$, form a Berge triangle. If $B_h\neq B_g$ and $B_h\cap B_g\neq\varnothing$, choose $z\in B_h\cap B_g$, $x\in B_h\setminus\{z\}$, and $y\in B_g\setminus\{x,z\}$; then $e_i,h,g$ form a Berge triangle on $x,y,z$. Thus the sets $B_h$ are distinct and pairwise disjoint, which proves the bound.

Remove the edges of $M$ and every assigned edge $h$ with $|B_h|\geq3$, and denote the family of remaining edges by $\mathcal H'$. Then
\begin{equation}\label{eq:rs-remaining-edges}
|\mathcal H'|\geq e(\mathcal H)-\left(1+\left\lfloor\frac{r}{3}\right\rfloor\right)\nu.
\end{equation}
For each $i\in[\nu]$, let $\mathcal M_i=\{h\in\mathcal H':\iota(h)=i\}$ and let $\mathcal K_i$ be the subhypergraph of $\mathcal H$ with edge set $\{e_i\}\cup\mathcal M_i$. The families $\mathcal M_i$ partition $\mathcal H'$, and every $h\in\mathcal M_i$ satisfies $1\leq|B_h|\leq2$. Every $\mathcal K_i$ is Berge-$K_3$-free.

Fix $i\in[\nu]$ and apply Lemma~\ref{lem:link-components} to $L_{\mathcal K_i}(e_i)$. If $h\in\mathcal M_i$ has $|B_h|=2$, set $U_h=A_h$, so $|U_h|=r-2\geq3$. If $|B_h|=1$ and $A_h$ forms a one-edge component, set $U_h=A_h$. If $A_h$ belongs to a component isomorphic to $S_p^{(r-1)}$ with common vertex $z$, set $U_h=A_h\setminus\{z\}$ and $|U_h|=r-2\geq3$. Lemma~\ref{lem:link-components} implies that, within a fixed $\mathcal M_i$, the sets $U_h$ obtained in these cases are pairwise disjoint and are disjoint from the vertex sets of the components isomorphic to $T_{2,q}^{(r-1)}$.

Consider a component of $L_{\mathcal K_i}(e_i)$ isomorphic to $T_{2,q}^{(r-1)}$, with edges $A_f$ and $A_g$, where $f,g\in\mathcal M_i$. Put $I=A_f\cap A_g$, $P=A_f\setminus I$, and $Q=A_g\setminus I$. Since $A_f$ and $A_g$ are distinct sets of the same size, both $P$ and $Q$ are nonempty. Choose distinct $x,y\in I$, together with $u\in P$ and $v\in Q$, and set $p(f)=\{x,u\}$ and $p(g)=\{y,v\}$. These pairs are vertex-disjoint and private to their respective edges in $\mathcal H$. Indeed, if an edge $k\neq f$ contained $\{x,u\}$, then $k\neq g$ because $u\notin A_g$, and the distinct vertices $y,x,u$ together with the edges $f,g,k$ would form a Berge triangle. The proof for $p(g)$ is symmetric.

Let $\mathcal T$ be the family of edges of $\mathcal H'$ belonging to components isomorphic to $T_{2,q}^{(r-1)}$. The pairs already chosen for the edges of $\mathcal T$ are private and hence pairwise distinct. Process the edges of $\mathcal H'\setminus\mathcal T$ in an arbitrary order. When processing $h$, choose a pair private to $h$ from $\binom{U_h}{2}$ if one exists. Otherwise, Lemma~\ref{private_pair_core_claim} supplies the edge $\tau(h,U_h)$. If a previously chosen pair $p(g)$ lies in $\binom{U_h}{2}$, then both $h$ and $g$ contain $p(g)$, so $g=\tau(h,U_h)$. Hence at most one previously chosen pair lies in $\binom{U_h}{2}$. Since $|U_h|\geq3$, choose $p(h)\in\binom{U_h}{2}$ different from that possible pair. Induction over the order shows that $h\mapsto p(h)$ is injective on $\mathcal H'$.

For each $i\in[\nu]$, let $M_i$ be the graph with edge set $\{p(h):h\in\mathcal M_i\}$. The pairs chosen from the two edges of each component isomorphic to $T_{2,q}^{(r-1)}$ are vertex-disjoint; pairs arising from different components use disjoint vertex sets; and all other pairs lie in the pairwise disjoint sets $U_h$. Hence every $M_i$ is a matching.

Let $G$ be the simple graph on $V(\mathcal H)$ with edge set $E(G)=\{p(h):h\in\mathcal H'\}$. Injectivity gives $e(G)=|\mathcal H'|$. For each $h\in\mathcal M_i$, choose a vertex $a_h\in B_h$ and color $p(h)$ by $(i,a_h)$. There are at most $r\nu\leq rs$ colors, and every color class is a matching because it is contained in $M_i$.

Each color class is induced. Suppose that distinct graph edges $p(h)$ and $p(g)$ have the same color $(i,a)$ and that an additional edge $p(k)=\{x,y\}$ of $G$ joins their endpoint sets, where $x\in p(h)$ and $y\in p(g)$. Since $M_i$ is a matching and $h\mapsto p(h)$ is injective, the hyperedges $h,g,k$ are distinct. We have $a\in h\cap g$, $x\in h\cap k$, and $y\in g\cap k$. Moreover, $a\in e_i$, while $x,y\notin e_i$, and $x\neq y$ because $p(h)$ and $p(g)$ are disjoint. Thus $a,x,y$ are distinct. The pairs $ax,xy,ya$ lie in the distinct hyperedges $h,k,g$, respectively, and form a Berge triangle, a contradiction.

It follows from \eqref{eq:rs-remaining-edges} that $e(G)=|\mathcal H'|\geq e(\mathcal H)-(1+\lfloor r/3\rfloor)\nu\geq e(\mathcal H)-(1+\lfloor r/3\rfloor)s$, and $E(G)$ is partitioned into at most $r\nu\leq rs$ induced matchings. By the definition of $\operatorname{RM}(n,rs)$, $e(\mathcal H)\leq\operatorname{RM}(n,rs)+(1+\lfloor r/3\rfloor)s$. Taking the maximum over all admissible $\mathcal H$ proves the consequent extremal inequality.
\end{proof}

Theorem~\ref{RS_reduction_theorem} follows from Theorem~\ref{RS_reduction_lemma}.

 \section*{Acknowledgment}

The work is partly supported by the National Natural Science Foundation of China (Nos. 12371354, W2521102), the Science and Technology Commission of Shanghai Municipality (No. 25LN3200600) and the Montenegrin-Chinese Science and Technology Cooperation Project (No. 4-3).

\bibliographystyle{abbrv}
\bibliography{reference}

@article{Gerbner_and_Palmer_Def_of_BG_first,
author = {Gerbner, D\'{a}niel and Palmer, Cory},
title = {Extremal Results for Berge Hypergraphs},
journal = {SIAM Journal on Discrete Mathematics},
volume = {31},
number = {4},
pages = {2314-2327},
year = {2017},
doi = {10.1137/16M1066191},
URL = {https://doi.org/10.1137/16M1066191}
}

@article{keevash2011hypergraph,
  title={Hypergraph {Tur\'an} problems},
  author={Keevash, Peter},
  journal={Surveys in combinatorics},
  volume={392},
  pages={83--140},
  year={2011},
  publisher={Cambridge University Press Cambridge}
}

@article{T41,
  title={On an extremal problem in graph theory},
  author={Tur{\'a}n, P},
  journal={Mat. Fiz. Lapok},
  volume={48},
  pages={436--452},
  year={1941}
}

@incollection{FurediSimonovits2013,
  author = {F{\"u}redi, Z. and Simonovits, M.},
  title = {The history of degenerate (bipartite) extremal graph problems},
  booktitle = {Erd{\H{o}}s centennial},
  pages = {169--264},
  year = {2013},
  publisher = {Springer Berlin Heidelberg},
  address = {Berlin, Heidelberg}
}

@article {ErdosGallai59,
    AUTHOR = {Erd\H{o}s, P. and Gallai, T.},
     TITLE = {On maximal paths and circuits of graphs},
   JOURNAL = {Acta Math. Acad. Sci. Hungar.},
  FJOURNAL = {Acta Mathematica. Academiae Scientiarum Hungaricae},
    VOLUME = {10},
      YEAR = {1959},
     PAGES = {337--356},
      ISSN = {0001-5954},
   MRCLASS = {05.00},
  MRNUMBER = {114772},
MRREVIEWER = {W. T. Tutte},
       DOI = {10.1007/BF02024498},
       URL = {https://doi.org/10.1007/BF02024498},
}

@article {AlonFrankl24JCTB_K_M,
    AUTHOR = {Alon, Noga and Frankl, P\'eter},
     TITLE = {Tur\'an graphs with bounded matching number},
   JOURNAL = {J. Combin. Theory Ser. B},
  FJOURNAL = {Journal of Combinatorial Theory. Series B},
    VOLUME = {165},
      YEAR = {2024},
     PAGES = {223--229},
      ISSN = {0095-8956,1096-0902},
   MRCLASS = {05C35 (05C70)},
  MRNUMBER = {4678741},
MRREVIEWER = {Narayanan\ Narayanan},
       DOI = {10.1016/j.jctb.2023.12.002},
       URL = {https://doi.org/10.1016/j.jctb.2023.12.002},
}

@article{ErdosSimonovits1966,
  author = {Erd\H{o}s, P. and Simonovits, M.},
  title = {A limit theorem in graph theory},
  journal = {Studia Scientiarum Mathematicarum Hungarica},
  volume = {1},
  year = {1966},
  pages = {51--57}
}

@article{ErdosStone1946,
  author = {Erd\H{o}s, P. and Stone, A. H.},
  title = {On the structure of linear graphs},
  journal = {Bulletin of the American Mathematical Society},
  volume = {52},
  year = {1946},
  pages = {1087--1091}
}

@article{G2024_Turan_bounded_matching,
  title={On {T}ur{\'a}n problems with bounded matching number},
  author={Gerbner, D{\'a}niel},
  journal={J. Graph Theory},
  volume={106},
  number={1},
  pages={23--29},
  year={2024},
  publisher={Wiley Online Library}
}

@article{xue2024_Turan_bounded_matching,
  title={On generalized {T}ur{\'a}n problems with bounded matching number},
  author={Xue, Yisai and Kang, Liying},
  journal={arXiv preprint arXiv:2410.12338},
  year={2024}
}

@article{zhao_and_Lu2024_Turan_bounded_matching,
  title={Generalized {T}ur{\'a}n problems for a matching and long cycles},
  author={Zhao, Xiamiao and Lu, Mei},
  journal={arXiv preprint arXiv:2412.18853},
  year={2024}
}

@article{zhu2025_Turan_bounded_matching,
  title={Extremal problems for a matching and any other graph},
  author={Zhu, Xiutao and Chen, Yaojun},
  journal={J. Graph Theory},
  volume={109},
  number={1},
  pages={19--24},
  year={2025},
  publisher={Wiley Online Library}
}

@article{GTZ2025,
  title={On hypergraph {T}ur{\'a}n problems with bounded matching number},
  author={Gerbner, D{\'a}niel and Tompkins, Casey and Zhou, Junpeng},
  journal={European J. Combin.},
  volume={127},
  pages={104155},
  year={2025},
  publisher={Elsevier}
}

@article{WWY2025,
  title={Hypergraph {T}ur{\'a}n problem of the generalized triangle with bounded matching number},
  author={Wang, Jian and Wang, Wenbin and Yang, Weihua},
  journal={arXiv preprint arXiv:2507.04579},
  year={2025}
}

@article{CLQY25,
  title={Triple systems with bounded matching number: some constructions and exact {T}ur\'{a}n number},
  author={Nannan Chen and Miao Liu and Yuzhen Qi and Caihong Yang},
  journal={arXiv preprint arXiv:2511.17000},
  year={2025}
}

@article{YangZengZhang2025,
  author = {Yang, C. and Zeng, J. and Zhang, X. D.},
  title = {A hypergraph analogue of {Alon--Frankl} Theorem},
  journal = {arXiv preprint arXiv:2511.21096},
  eprint = {2511.21096},
  year = {2025}
}

@article{gerbner2020general,
  title={General lemmas for Berge--Tur{\'a}n hypergraph problems},
  author={Gerbner, D{\'a}niel and Methuku, Abhishek and Palmer, Cory},
  journal={European Journal of Combinatorics},
  volume={86},
  pages={103082},
  year={2020},
  publisher={Elsevier}
}

@article{ergemlidze2020avoiding,
  title={Avoiding long {B}erge cycles: the missing cases $k= r+ 1$ and $k= r+ 2$},
  author={Ergemlidze, Beka and Gy{\H{o}}ri, Ervin and Methuku, Abhishek and Salia, Nika and Tompkins, Casey and Zamora, Oscar},
  journal={Combinatorics, Probability and Computing},
  volume={29},
  number={3},
  pages={423--435},
  year={2020},
  publisher={Cambridge University Press}
}

@article{zhu2020_Berge_K4_turan,
  title={The {T}ur\'an {N}umber of {B}erge-${K}_4$ in 3-{U}niform {H}ypergraphs},
  author={Zhu, Hui and Kang, Liying and Ni, Zhenyu and Shan, Erfang},
  journal={SIAM Journal on Discrete Mathematics},
  volume={34},
  number={3},
  pages={1485--1492},
  year={2020},
  publisher={SIAM}
}

@article{furedi2019_long_Berge_cycle,
  title={Avoiding long Berge cycles},
  author={F{\"u}redi, Zolt{\'a}n and Kostochka, Alexandr and Luo, Ruth},
  journal={Journal of Combinatorial Theory, Series B},
  volume={137},
  pages={55--64},
  year={2019},
  publisher={Elsevier}
}

@article{gyHori2012Berge_cycle_given_length,
  title={Hypergraphs with no cycle of a given length},
  author={Gy{\H{o}}ri, Ervin and Lemons, Nathan},
  journal={Combinatorics, Probability and Computing},
  volume={21},
  number={1-2},
  pages={193--201},
  year={2012},
  publisher={Cambridge University Press}
}

@article{grosz2020uniformity_threshold,
  title={Uniformity thresholds for the asymptotic size of extremal Berge-{F}-free hypergraphs},
  author={Gr{\'o}sz, D{\'a}niel and Methuku, Abhishek and Tompkins, Casey},
  journal={European Journal of Combinatorics},
  volume={88},
  pages={103109},
  year={2020},
  publisher={Elsevier}
}

@article{naor2005_bipartite_turan,
  title={A note on bipartite graphs without 2k-cycles},
  author={Naor, Assaf and Verstra{\"e}te, Jacques},
  journal={Combinatorics, Probability and Computing},
  volume={14},
  number={5-6},
  pages={845--849},
  year={2005},
  publisher={Cambridge University Press}
}

@article{gyHori2006triangle,
  title={Triangle-free hypergraphs},
  author={Gy{\H{o}}ri, Ervin},
  journal={Combinatorics, Probability and Computing},
  volume={15},
  number={1-2},
  pages={185--191},
  year={2006},
  publisher={Cambridge University Press}
}

@article{zhou2025Berge_forest,
  title={On Tur{\'a}n problems for Berge forests},
  author={Zhou, Junpeng and Gerbner, D{\'a}niel and Yuan, Xiying},
  journal={arXiv preprint arXiv:2506.16140},
  year={2025}
}

@article{gyarfas2019Berge_K4,
  title={The {Tur\'an} Number of {Berge $K_4$} in Triple Systems},
  author={Gy{\'a}rf{\'a}s, Andr{\'a}s},
  journal={SIAM Journal on Discrete Mathematics},
  volume={33},
  number={1},
  pages={383--392},
  year={2019},
  publisher={SIAM}
}

@article{gerbner2024Berge_book,
  title={The Tur{\'a}n number of {Berge} book hypergraphs},
  author={Gerbner, D{\'a}niel},
  journal={SIAM Journal on Discrete Mathematics},
  volume={38},
  number={4},
  pages={2896--2912},
  year={2024},
  publisher={SIAM}
}

@article{zhan2025tur,
  title={Tur{\'a}n number of disjoint Berge paths},
  author={Zhan, Yiyan and Zhao, Xiamiao and Lu, Mei},
  journal={arXiv preprint arXiv:2512.23382},
  year={2025}
}

@article{frankl2024triangle,
title = {Triangle-Free Triple Systems},
author = {Frankl, Peter and F{\"u}redi, Zolt{'a}n and
Goorevitch, Ido and Holzman, Ron and Simonyi, G{'a}bor},
journal = {Electronic Journal of Combinatorics},
volume = {33},
number = {2},
pages = {P2.34},
year = {2026},
doi = {10.37236/14115}
}

@article{bondyS_1974_even_cycles,
  title={Cycles of even length in graphs},
  author={Bondy, John A and Simonovits, Mikl{\'o}s},
  journal={Journal of Combinatorial Theory, Series B},
  volume={16},
  number={2},
  pages={97--105},
  year={1974},
  publisher={Elsevier}
}

@incollection{verstraete2016extremalC2k_survey,
  title={Extremal problems for cycles in graphs},
  author={Verstra{\"e}te, Jacques},
  booktitle={Recent trends in combinatorics},
  pages={83--116},
  year={2016},
  publisher={Springer}
}

@article{gyHori2025linear,
  title={Linear three-uniform hypergraphs with no {Berge} path of given length},
  author={Gy{\H{o}}ri, Ervin and Salia, Nika},
  journal={Journal of Combinatorial Theory, Series B},
  volume={171},
  pages={36--48},
  year={2025},
  publisher={Elsevier}
}

@article{fox_huang_sudakov_RS,
  author = {Fox, Jacob and Huang, Hao and Sudakov, Benny},
  title = {On graphs decomposable into induced matchings of linear sizes},
  journal = {Bulletin of the London Mathematical Society},
  volume = {49},
  number = {1},
  pages = {45--57},
  year = {2017},
  doi = {10.1112/blms.12005}
}

@article{FRANKL20131068,
title = {Improved bounds for {Erd\H{o}s}' Matching Conjecture},
journal = {Journal of Combinatorial Theory, Series A},
volume = {120},
number = {5},
pages = {1068-1072},
year = {2013},
issn = {0097-3165},
doi = {10.1016/j.jcta.2013.01.008},
author = {Peter Frankl},
}

@article{GLSZ2021_structure_long_Berge_cycles,
  title={The structure of hypergraphs without long {B}erge cycles},
  author={Gy{\H{o}}ri, Ervin and Lemons, Nathan and Salia, Nika and Zamora, Oscar},
  journal={J. Combin. Theory Ser. B},
  volume={148},
  pages={239--250},
  year={2021},
  publisher={Elsevier},
  doi={10.1016/j.jctb.2020.04.007},
  url={https://doi.org/10.1016/j.jctb.2020.04.007}
}

@article{FKL2020_Berge_cycles_non_uniform,
  title={Berge cycles in non-uniform hypergraphs},
  author={F{\"u}redi, Zolt{\'a}n and Kostochka, Alexandr and Luo, Ruth},
  journal={Electron. J. Combin.},
  volume={27},
  number={3},
  pages={P3.9},
  year={2020},
  doi={10.37236/9346},
  url={https://doi.org/10.37236/9346}
}

@article{GSZ2021_connected_Berge_paths,
  title={Connected hypergraphs without long {B}erge-paths},
  author={Gy{\H{o}}ri, Ervin and Salia, Nika and Zamora, Oscar},
  journal={European J. Combin.},
  volume={96},
  pages={103353},
  year={2021},
  publisher={Elsevier},
  doi={10.1016/j.ejc.2021.103353},
  url={https://doi.org/10.1016/j.ejc.2021.103353}
}

\end{document}